\documentclass[onecolumn, draftcls]{IEEEtran}
\usepackage{amsmath,amsfonts,amssymb,epsfig,color,accents,psfrag}
\usepackage[tight,footnotesize]{subfigure}
\usepackage{lineno}
\usepackage{anysize}
\newtheorem{assumption}{Assumption}
\newtheorem{example}{Example}
\newtheorem{theorem}{Theorem}
\newtheorem{proposition}{Proposition}
\newtheorem{lemma}{Lemma}
\newtheorem{remark}{Remark}
\newcommand{\dspace}{\vspace{-0.5cm}}

\newcommand{\ie}{i.e.}

\DeclareMathAccent{\wtilde}{\mathord}{largesymbols}{"65}
\newcommand{\uest}[1]{\underaccent{\wtilde}{#1}}
\newcommand{\oest}[1]{\widetilde{#1}}

\marginsize{20mm}{20mm}{10mm}{35mm}
\ifCLASSINFOpdf
\else
\fi

\begin{document}
%
\title{Accelerated Gradient Methods\\ for Networked Optimization}
%
%
%

\author{Euhanna Ghadimi,
        ~Iman Shames,
        ~and~Mikael Johansson
\thanks{E.~Ghadimi and M.~Johansson are with the ACCESS Linnaeus Center, Electrical Engineering, Royal Institute of Technology, Stockholm, Sweden.
{\tt\small \{euhanna, mikaelj\}@ee.kth.se}. 
 I.~Shames is with the Department of Electrical and Electronic Engineering, The University of Melbourne, Melbourne, Australia.
{\tt\small iman.shames@unimelb.edu.au}.}}

\maketitle

\begin{abstract}
We develop multi-step gradient methods for
network-constrained optimization of strongly convex functions with Lipschitz-continuous gradients. Given the topology of the underlying network and bounds on the Hessian of the objective function, we determine the algorithm parameters that guarantee the fastest convergence and characterize situations when significant speed-ups can be obtained over the standard gradient method. Furthermore, we quantify how the performance of the gradient method and its accelerated counterpart are affected by uncertainty in the problem data, and conclude that in most cases our proposed method outperforms gradient descent. Finally, we apply the proposed technique to three engineering problems: resource allocation under network-wide budget constraints, distributed averaging, and Internet congestion control. In all cases, we demonstrate that our algorithm converges more rapidly than alternative algorithms reported in the literature.
\end{abstract}

%

%
\IEEEpeerreviewmaketitle

\section{Introduction}

\IEEEPARstart{D}{istributed} optimization has recently attracted significant attention from several  research communities. Examples include the work on network utility maximization for resource allocation in communication networks~\cite{kelly98}, distributed coordination of multi-agent systems~\cite{olfati2007}, collaborative estimation in wireless sensor networks~\cite{Kar12}, distributed machine learning~\cite{Boyd11}, and many others. The majority of these praxes apply gradient or sub-gradient methods to the dual formulation of the decision problem. Although gradient methods are easy to implement and require modest computations, they suffer from slow convergence. In some cases, such as the development of distributed power control algorithms for cellular phones~\cite{CellularPower00}, one can replace gradient methods by fixed-point iterations and achieve improved convergence rates. For other problems, such as average consensus~\cite{Tsitsiklis86}, a number of heuristic methods have  been proposed that improve the convergence time of the standard method~\cite{Cao06,bjornThesis}.  However, we are not interested in tailoring techniques to individual problems; our aim is to develop general-purpose schemes that retain the simplicity of the gradient method, yet improve the convergence factors.

Even if the optimization problem is convex and the subgradient method is guaranteed to converge to an optimal solution, the rate of convergence is very modest.
The convergence rate of the gradient method is improved
if the objective function is differentiable with Lipschitz-continuous gradient, and even more so if the function is also strongly convex. However, for smooth optimization problems several techniques allow for even better convergence rates. One such technique is higher-order methods, such as Newton's method~\cite{Bertsekas:nonlinear}, which use both the gradient and the Hessian of the objective function. Although distributed Newton methods have recently been developed for special problem classes (\emph{e.g.}, \cite{damiano11,Asu011}), they impose large communication overhead to collect global Hessian information.  Another technique is the augmented Lagrangian dual method~\cite{Bertsekas1989}. This method was originally developed to cope with robustness issues of the dual ascent method, but it turns out that different variations of this technique, such as the method of multipliers~\cite{Boyd11}, tend to converge in fewer iterations than gradient descent. Recently a few applications of these algorithms to distributed optimization have been proposed~\cite{Boyd11, ADMMConsensus11}  but convergence rate estimates and optimal algorithm parameters are still unaddressed. A third way to obtain faster convergence is to use \emph{multi-step methods}~\cite{polyak,Bertsekas:nonlinear}. These methods rely only on gradient information but use a history of the past iterates when computing the future ones. This paper explores the latter approach for distributed optimization, and addresses the design, convergence properties, optimal step-size selection, and  robustness of networked multi-step methods. Moreover, we also apply the developed techniques to three important classes of distributed optimization problems.

This paper makes the following contributions. First, we develop an multi-step weighted gradient method that maintains a network-wide constraint on the decision variables throughout the iterations.
The accelerated algorithm is based on the \emph{heavy ball} method by Polyak~\cite{polyak} extended to the networked setting. We derive optimal algorithm parameters, show that the method has linear convergence rate and quantify the improvement in convergence factor over the gradient method. Our analysis shows that method is particularly advantageous when the eigenvalues of the Hessian of the objective function and/or the eigenvalues of the graph Laplacian of the underlying network have a large spread.
Second, we investigate how similar techniques can be used to accelerate dual decomposition across a network of decision-makers. In particular, given smoothness parameters of the objective function, we present closed-form expressions for the optimal parameters of an accelerated gradient method for the dual. Third, we quantify how the convergence properties of the algorithm are affected when the algorithm is tuned using misestimated problem parameters. This robustness analysis shows that the accelerated algorithm endures parameter violations well and in most cases outperforms its non-accelerated counterpart. Finally, we apply the developed algorithms to three case studies: networked resource allocation, consensus, and network flow control. In each application we demonstrate superior performance compared to alternatives from the literature.

The paper is organized as follows. In Section~\ref{sec:problem_formulation}, we introduce our networked optimization problem. Section~\ref{sec:background} reviews multi-step gradient techniques. Section~\ref{sec:weighted_gradient} proposes a multi-step weighted gradient algorithm, establishes conditions for its convergence and derives optimal step-size parameters. Section~\ref{sec:multistep_dual} develops a technique for accelerating the dual problem based on parameters for the (smooth) primal.  Section~\ref{sec:robustness} presents a robustness analysis of the multi-step algorithm in the presence of uncertainty.
Section~\ref{sec:applications} applies the proposed techniques to three engineering problems: resource allocation, consensus and network flow control; numerical results and performance comparisons are presented for each case study. Finally, concluding remarks are given in  Section~\ref{sec:conclusion}.

\section{Assumptions and problem formulation}
\label{sec:problem_formulation}
This paper is concerned with collaborative optimization by a network of decision-makers. Each decision-maker $v$ is endowed with a loss function $f_v:\mathbf{R}\mapsto {\mathbf R}$,  has control of one decision-variable $x_v\in {\mathbf R}$, and collaborates with the others to solve the optimization problem
\begin{align}
	\begin{array}[c]{ll}
	\mbox{minimize} & \sum_{v\in {\mathcal V}} f_v(x_v)\\
	\mbox{subject to} & Ax=b
	\end{array} \label{eqn:problem_formulation}
\end{align}
for given matrices $A \in \mathbf{R}^{m\times n}$ and $b\in {\mathbf R}^m$. We will assume that $b$ lies in the range space of $A$, \emph{i.e.} that there exists at least one decision vector $x$ that satisfies the constraints.

The information exchange between decision-makers is represented by a graph ${\mathcal G}=({\mathcal V}, {\mathcal E})$ with vertex set ${\mathcal V}=\{1,2,\dots, n\}$ and edge set ${\mathcal E}\subseteq {\mathcal V}\times{\mathcal V}$.
%
Specifically, at each time $t$, we will assume that decision-maker $v$ has access to
$\nabla f_w(x_w(t))$ for all its neighbors $w\in {\mathcal N}_v \triangleq \{w \;\vert\; (v,w)\in {\mathcal E}\}$.

Most acceleration techniques in the literature (\emph{e.g.}~\cite{Nesterov04,Nesterov05,Devolder11_Double}) require that the loss functions are smooth and convex. Similarly, we will make the following assumptions:
\begin{assumption} \label{ass:objective_bounds}
Each loss function $f_v$ is convex and twice continuously differentiable with
\begin{align}
l_v \leq \nabla^2 f_v(x_v) \leq u_v, \quad \forall x_v
\label{eq:Hessian_Bounds}
\end{align}
for some positive real constants $l_v, u_v$ with $0<l_v\leq u_v$.
\end{assumption}

Some remarks are in order. Let $l=\min_{v\in {\mathcal V}} l_{v}$, $u=\max_{v\in {\mathcal V}} u_v$ and define $f(x){\mathrel{\mathop:}=} \sum_{v\in {\mathcal V}} f_v(x_v)$. Then, Assumption~\ref{ass:objective_bounds} ensures that $f(x)$ is strongly convex with modulus $l$:
\begin{align*}
	f(y) &\geq f(x) +(y-x)^{\top}\nabla f(x) + \frac{l}{2}\Vert y-x\Vert^2 \quad\forall (x,y)
\end{align*}
and that its gradient is Lipschitz-continuous with constant $u$:
\begin{align*}
	f(y) &\leq f(x)+(y-x)^{\top}\nabla f(x)+\frac{u}{2}\Vert y-x\Vert^2 \quad \forall (x,y)
\end{align*}
See, \emph{e.g},~\cite[Lemma 1.2.2 and Theorem 2.1.11]{Nesterov04} for details. Similarly, the Hessian of $f$ satisfies
\begin{align}
	lI &\leq \nabla^2f(x)\leq uI \quad \forall x \label{eqn:hessian_bounds}
\end{align}
Furthermore, Assumption~\ref{ass:objective_bounds} guarantees that (\ref{eqn:problem_formulation}) is a convex optimization problem whose unique optimizer $x^{\star}$ satisfies
\begin{align}
	Ax^{\star} &= b, \qquad \nabla f(x^{\star})=A^{\top}\mu^{\star}
\label{eqn:fixed_point_feasibility_cond}
\end{align}
where $\mu^{\star}\in \mathbb{R}^m$ is the (unique) optimal Lagrange multiplier for the linear constraints.

\section{Background on multi-step methods}
\label{sec:background}
The basic gradient method for unconstrained minimization of a convex function $f(x)$ takes the form
\begin{align}
	x(k+1) &= x(k)-\alpha \nabla f(x(k)) \label{eqn:basic_gradient},
\end{align}
where $\alpha>0$ is a fixed step-size parameter. Assume that $f(x)$ is strongly convex with modulus $l$ and has Lipschitz-continuous gradient with constant $u$. Then if $\alpha<2/u$, the sequence $\{ x(k)\}$ generated by  (\ref{eqn:basic_gradient}) converges to $x^{\star}$ at linear rate, \emph{i.e.} there exists a convergence factor $q\in (0,1)$ such that
\[
\Vert x(k+1)-x^{\star}\Vert \leq q \Vert x(k)-x^{\star}\Vert
\] for all $k$. The smallest convergence factor is $q=(u-l)/(u+l)$ obtained for the step-size $\alpha=2/(l+u)$~(see, \emph{e.g.},~\cite{polyak}).

While the convergence rate cannot be improved unless higher-order information is considered~\cite{polyak}, the convergence factor $q$ can be meliorated by accounting for the history of iterates when computing the ones to come. Methods in which the next iterate depends not only on the current iterate but also on the preceding ones are called \emph{multi-step methods}. The simplest multi-step extension of the gradient method is
\begin{align}
	x(k+1) &= x(k)-\alpha \nabla f(x(k)) + \beta \left( x(k)-x(k-1) \right) \label{eqn:multi_step}
\end{align}
for fixed step-size parameters $\alpha>0$ and $\beta>0$. This technique, originally proposed by Polyak, is sometimes called the heavy-ball method based on the physical interpretation of the added ``momentum term''.  For a centralized set-up, Polyak derived the optimal step-size parameters and showed that these guaranteed a convergence factor of $(\sqrt{u}-\sqrt{l})/(\sqrt{u}+\sqrt{l})$ , which is always smaller than the convergence factor for the gradient method and significantly so when $\sqrt{u}/\sqrt{l}$ is large.

In the following sections, we will develop multi-step gradient methods for network-constrained optimization, analyze their convergence properties, and develop techniques for finding the optimal algorithm parameters.

\section{A multi-step weighted gradient method} \label{sec:weighted_gradient}
In the absence of constraints, (\ref{eqn:problem_formulation}) is trivial to solve since the objective function is separable and each decision-maker could simply minimize its loss independently of the others. Hence, it is the existence of constraints that makes (\ref{eqn:problem_formulation}) challenging. In the optimization literature, there are essentially two ways of dealing with constraints. One way is to project the iterates onto the constraint set to maintain feasibility at all times; such a method will be developed in this section. The other way is to use dual decomposition to eliminate couplings between decision-makers and solve the associated dual problem; we will return to such techniques in Section~\ref{sec:multistep_dual}.

Computing the Euclidean projection onto the constraint of (\ref{eqn:problem_formulation}) typically requires the full decision vector $x$, which is not available to the decision-makers in our setting. An alternative, explored \emph{e.g.} in~\cite{XiB:06}, is to consider \emph{weighted gradient methods} which use a linear combination of the information available to nodes to ensure that iterates remain feasible. For our problem (\ref{eqn:problem_formulation}) the weighted gradient method takes the form
\begin{align}
	x(k+1) &= x(k)-\alpha W \nabla f(x(k)) \label{eqn:scaled_gradient_iteration}
\end{align}
where $W\in \mathbf{R}^{n\times n}$ is a weight matrix that satisfies the sparsity constraint that $W_{vw}=0$ if $v\neq w$ and $(v,w)\not\in {\mathcal E}$. In this way, the iterations (\ref{eqn:scaled_gradient_iteration}) read
\begin{align*}
	x_v(k+1) &= x_v(k) - \alpha \sum_{w\in v\cup {\mathcal N}_v} W_{vw} \nabla f_w(x_w(k))
\end{align*}
and can be executed by individual decision-makers based on the information that they have access to. If $W$ satisfies
\begin{align}
	AW &= 0 && WA^{\top}=0 \label{eqn:weight_constraints}
\end{align}
then any initially feasible $x(0)$ will always remain feasible. While the constraints on $W$ might appear restrictive, it is possible to construct appropriate weight matrices for many applications. The following examples describe two such cases.
\begin{example} \label{ex:total_budget_constraint}
When the decision-makers are only constrained by a total resource budget, (\ref{eqn:problem_formulation}) reduces to
\begin{align*}
	\begin{array}[c]{ll}
	\mbox{minimize} & \sum_{v\in {\mathcal V}} f_v(x_v)\\
	\mbox{subject to} & \sum_{v\in {\mathcal V}} x_v = x_{\mathrm{tot}}
	\end{array}
\end{align*}
A distributed gradient method for this problem was developed in~\cite{HSS:80}. Later,~\cite{XiB:06} interpreted these as a weighted gradient method and developed techniques for computing the weight matrix $W$ that minimizes the guaranteed convergence factor.
\end{example}

\begin{example}\label{ex:consensus}
Consider a scenario where the decision-makers have to find a common decision $x$ that minimizes the total cost
\begin{align*}
	\begin{array}[c]{ll}
	\mbox{minimize} & \sum_{v\in {\mathcal V}} f_v(x)
	\end{array}
\end{align*}
We can rewrite this problem in our standard form (\ref{eqn:problem_formulation}) by introducing local decision variables $x_v$:
\begin{align}
	\begin{array}[c]{ll}
	\mbox{minimize} &\sum_{v\in {\mathcal V}} f_v(x_v)\\
	\mbox{subject to} & x_v=x_w \qquad\quad \forall (v,w)\in {\mathcal E}
	\end{array} \label{eqn:ex2_localx}
\end{align}
Note that in vector form, the constraint of (\ref{eqn:ex2_localx}) reads $Ax=0$ where $A\in {\mathbf R}^{\vert {\mathcal E} \vert \times \vert {\mathcal V}\vert}$ is the incidence matrix of the graph ${\mathcal G}$. Next, we will show that the gradient iterations for the dual problem of (\ref{eqn:ex2_localx}) has the structure of a weighted gradient method in the primal variables. To this end, we form the Lagrangian  $L(x,\mu)=f(x)-\mu^{\top}Ax$ and the dual function
\begin{align*}
	d(\mu) &= \inf_x L(x,\mu)=\inf_x f(x)-\mu^{\top}Ax
\end{align*}
Under Assumption~\ref{ass:objective_bounds}, the Lagrangian has a unique minimizer $x^{\star}(\mu)=(\nabla f)^{-1}(A^{\top}\mu)$ and the dual function is continuously differentiable with $\nabla d(\mu)=-Ax^{\star}(\mu)$. Hence, the iterations
\begin{align*}
	\mu(k+1) &= \mu(k)-\alpha Ax(k)\\
	x(k+1) &= \nabla f^{-1}(A^{\top}\mu(k+1))
\end{align*}
will converge to a primal-dual optimal pair for appropriately chosen step-size $\alpha$. Introducing $z(k) := A^{\top}x(k)$ and multiplying both sides of the iterations by $A^{\top}$, we obtain
\begin{align}
	\begin{array}[l]{lcr}
	z(k+1) = z(x)-\alpha W \nabla f^{-1}(z(k))\\
	x(k+1) = \nabla f^{-1}(z(k+1))
	\end{array} \label{eqn:ex2_dual_iterations}
\end{align}
Note that $W = A^{\top}A$ is the graph Laplacian of ${\mathcal G}$ and that $W$ satisfies the sparsity constraint for distributed execution detailed above. One can readily verify that $W$ has a simple eigenvalue at $0$ for which $W\mathbf{1}=\mathbf{0}$.

One important application of this technique is to distributed averaging, in which nodes should converge to the network-wide average of constants $c_v$ held by each node $v\in {\mathcal V}$. This average can be found by solving (\ref{eqn:ex2_localx}) with $f_v(x_v)=(x_v-c_v)^2/2$ (since its optimal solution is the average of the constants $c_v$). The corresponding iterations
 (\ref{eqn:ex2_dual_iterations}) read
\begin{align*}
	z(k+1) &= z(k)-\alpha W\left( z(k)-c\right)\\
	x(k+1) &= z(k+1) + c
\end{align*}
We will return to these iterations and their accelerated counterparts in Section~\ref{sec:applications}.
\end{example}

\subsection{A multi-step weighted gradient method and its convergence}
The examples indicate that variants of the weighted gradient method with improved convergence factors could also allow to speed up the convergence of network-wide resource allocation and consensus processes.
To this end, we consider the following multi-step variant of the weighted gradient iteration
\begin{align}
 x(k+1) &= x(k)-\alpha W \nabla f(x) + \beta\left( x(k)-x(k-1)\right) \label{eqn:multistep_scaled_gradient}
\end{align}
Under the sparsity constraint on $W$ detailed above, these iterations can be implemented by individual decision-makers. Moreover, (\ref{eqn:weight_constraints}) ensures that if $x(1)$ and $x(0)$ satisfy the constraints of (\ref{eqn:problem_formulation}) then every iterate produced by (\ref{eqn:multistep_scaled_gradient}) will also be feasible. The next theorem characterizes the convergence of the iterations (\ref{eqn:multistep_scaled_gradient}) and derive optimal step-size parameters.

\begin{theorem} \label{thm:multi_step}
Consider the optimization problem (\ref{eqn:problem_formulation}) under Assumption~\ref{ass:objective_bounds}, and let $x^{\star}$ denote its unique optimizer. Assume that $W$ has $m<n$ eigenvalue at $0$ and satisfies $AW=0$ and $WA^{\top}=0$. Let $H =\nabla^2 f(x^{\star})$ and $0=\lambda_1(WH)=\cdots =\lambda_{m}(WH)< \lambda_{m+1}(WH)= \underline{\lambda} \leq \cdots\leq \lambda_n(WH)=\overline{\lambda}$ be the magnitude of eigenvalues of $WH$. Then, for
\begin{align*}
	0&\leq \beta\leq 1, & 0<\alpha<\frac{2}{u}\frac{(1+\beta)}{\lambda_n(W)}
\end{align*}
the iterates (\ref{eqn:multistep_scaled_gradient}) converge to $x^{\star}$ at linear rate
\begin{align*}
	\Vert x(k+1)-x^{\star}\Vert &\leq q \Vert x(k)-x^{\star}\Vert \quad \forall k\geq 0
\end{align*}
with $q=\max\left\{ \sqrt{\beta}, \vert 1+\beta-\alpha\underline{\lambda}\vert-\sqrt{\beta},\vert 1+ \beta-\alpha\overline{\lambda}\vert-\sqrt{\beta}\right\}$.
Moreover, the minimal value of $q$ is
\begin{align*}
	q^{\star} &= \frac{\sqrt{\overline{\lambda}}-\sqrt{\underline{\lambda}}}{\sqrt{\overline{\lambda}}+\sqrt{\underline{\lambda}}}
\end{align*}
obtained for step-sizes $\alpha=\alpha^{\star}$ and $\beta=\beta^{\star}$ where
\begin{align*}
	\alpha^{\star} =\left(\frac{2}{\sqrt{\overline{\lambda}}+\sqrt{\underline{\lambda}}} \right)^2, \quad
\beta^{\star}=\left(\frac{\sqrt{\overline{\lambda}}-\sqrt{\underline{\lambda}}}{\sqrt{\overline{\lambda}}+\sqrt{\underline{\lambda}}} \right)^2
\end{align*}
\end{theorem}
\begin{proof}
See the appendix for all the proofs.
\end{proof}
Similar to the discussion in Section~\ref{sec:background}, it is interesting to investigate when (\ref{eqn:multistep_scaled_gradient}) significantly improves over the single-step algorithm. In~\cite{XiB:06}, it is shown that the best convergence factor of the weighted gradient iteration (\ref{eqn:scaled_gradient_iteration}) is
\begin{align*}
	q_0^{\star} &= \frac{\overline{\lambda} - \underline{\lambda}}{\overline{\lambda}+\underline{\lambda}}
\end{align*}
One can verify that $q^{\star}\leq q_0^{\star}$, \emph{i.e.} the multi-step iterations can always be tuned to converge faster. Moreover, the improvement in convergence factor depends on the quantity $\kappa=\overline{\lambda}/\underline{\lambda}$: when $\kappa$ is large, the speed-up is roughly proportional to $\sqrt{\kappa}$. In the networked setting, there are two reasons for a large value of $\kappa$. One is simply that the Hessian of the objective function is ill-conditioned, so that the ratio $u/l$ is large. The other is that the matrix $W$ is ill-conditioned, \emph{i.e.} that $\lambda_n(W)/\lambda_{m+1}(W)$ is large.  As we have seen in the examples, the graph Laplacian is often a valid choice for $W$. Thus, the topology of the underlying graph directly impacts the convergence rate (and the convergence rate improvements) of the multi-step weighted gradient method. We will return to this in detail in Section~\ref{sec:applications}.

In many applications, we will not know $H=\nabla^2 f(x^{\star})$, but only bounds such as (\ref{eqn:hessian_bounds}). The next result can then be useful
\begin{proposition} \label{prop:multistep_restricted}
Let $\underline{\lambda}_W=l\lambda_{m+1}(W)$ and $\overline{\lambda}_W=u \lambda_n(W)$. Then $\underline{\lambda}_W\leq \underline{\lambda}$ and $\overline{\lambda}_W\geq \overline{\lambda}$. Moreover, the step-sizes
\begin{align*}
\alpha&= \left( \frac{2}{\sqrt{\overline{\lambda}_W}+\sqrt{\underline{\lambda}_W}}\right)^2, \quad \beta=
\left(\frac{\sqrt{\overline{\lambda}_W}-\sqrt{\underline{\lambda}_W}}{\sqrt{\overline{\lambda}_W}+\sqrt{\underline{\lambda}_W}} \right)^2
\end{align*}
guarantee that~\eqref{eqn:multistep_scaled_gradient} converges to $x^{\star}$ at linear rate
\begin{align*}
	\Vert x(k+1)-x^{\star}\Vert &\leq \tilde{q} \Vert x(k)-x^{\star}\Vert \quad \forall k,
\end{align*}
where
\begin{align*}
	\tilde{q} &= \frac{\sqrt{\overline{\lambda}_W}-\sqrt{\underline{\lambda}_W}}{\sqrt{\overline{\lambda}_W}+\sqrt{\underline{\lambda}_W}}
\end{align*}
\end{proposition}

\subsection{Optimal weight selection for the multi-step method}
The results in the previous subsection provide optimal step-size parameters $\alpha$ and $\beta$ for a given weight matrix $W$. However, the expressions for the associated convergence factors depend on the eigenvalues of $WH$ and optimizing the entries in $W$ jointly with the step-size parameters can yield even further speed-ups. We make the following observation.

\begin{proposition}
\label{prop:multistep_condition_number_minimization}
Under the hypotheses of Proposition~\ref{prop:multistep_restricted},
\begin{itemize}
\item[(i)] If $H$ is known, then minimizing the convergence factor $q^{\star}$ is equivalent to minimizing $\overline{\lambda}/\underline{\lambda}$.
\item[(ii)] If $H$ is not known, while $l$ and $u$ in~\eqref{eqn:hessian_bounds} are, then the weight matrix that minimizes $\tilde{q}$ is the one with minimal value of $\lambda_n(W)/\lambda_{m+1}(W)$.
\end{itemize}
\end{proposition}

The next result shows how the optimal weight selection for both scenarios can be found via convex optimization.

\begin{proposition} \label{prop:weight_optimization}
Let ${\mathcal M}$ be the span of real symmetric matrices with the sparsity pattern induced by ${\mathcal G}$, \emph{i.e.}
\begin{align*}
	{\mathcal M} &= \{ M \in {\mathcal S}^n  \;\vert\;  S_{vw}=0 \mbox{ if } v\neq w \mbox{ and} (v,w)\not\in {\mathcal E}\}.
\end{align*}
Then the problem of minimizing $\overline{\lambda}/\underline{\lambda}$ is equivalent to
\begin{align}
\label{eqn:weight_optimization}
\begin{array}{ll}
\underset{\omega}{\mbox{minimize}}& t\\
\mbox{subject to}& I_{n-m} \leq P^\top H^{1/2} \omega H^{1/2} P \leq t I_{n-m}\\
& H^{1/2} \omega H^{1/2} \in \mathcal{M}, \quad H^{1/2} \omega H^{1/2} V = \textbf{0},
\end{array}
\end{align}
where $V=[v_1, \cdots, v_m] \in \mathbf{R}^{n\times m}$ is the eigenvector space corresponding to the zero eigenvalues of $WH^{1/2}$ and $P=[p_1, p_2 \cdots p_{n-m}]\in \mathbf{R}^{n\times n-m}$ is a matrix of unit vectors spanning $V^\bot$.
\end{proposition}
Note that when we only want to minimize the condition number of $W$ subject to the structural constraints, we simply set $H=I$ in the formulation above.

\begin{remark}\label{rem:weight_scaling} The lower bound in~\eqref{eqn:weight_optimization} is rather arbitrary: any scaled matrix $\gamma W$ for $\gamma\in \mathbb{R}_+$ has the same condition number as $W$, and if if $\alpha^{\star}$ and $\beta^{\star}$ are the optimal step-sizes for the matrix $W$, then $\alpha=\alpha^{\star}/\gamma$ and $\beta=\beta^{\star}$ are optimal for $\gamma W$.
\end{remark}

\section{A multi-step dual ascent method}\label{sec:multistep_dual}

An alternative approach for solving (\ref{eqn:problem_formulation}) is to use Lagrange relaxation, \emph{i.e.} to introduce Lagrange multipliers $\mu\in \mathbf{R}^m$ for the equality constraints and solve the dual problem. The dual function associated with (\ref{eqn:problem_formulation}) is then
\begin{align}
	d(\mu) &\triangleq \inf_x\, f(x)+\mu^{\top}(Ax-b) = -f_{\star}(-A^{\top}\mu)-\mu^{T}b \label{eqn:dual_function}
\end{align}
where~$f_{\star}(y) \triangleq \sup_x\; y^{\top}x-f(x)$
is the conjugate function of $f$. The dual problem is to maximize the dual function with respect to $\mu$, \ie,
\begin{align*}
	\begin{array}[c]{ll} \underset{\mu}{\mbox{minimize}} & -d(\mu) = f_{\star}(-A^{\top}\mu)+b^{\top}\mu \end{array}.
\end{align*}
Recall that if $f$ is strongly convex then $f_{\star}$ and hence $-d$ are convex and continuously differentiable~\cite{Van:12}. Hence, in light of our earlier discussion, it is natural to attempt to solve the dual problem using the multi-step iteration
\begin{align}
\label{eqn:dual_accelerated_iterations}
	\mu(k+1)&= \mu(k)+\alpha \nabla d(\mu(k))+\beta( \mu(k)-\mu(k-1)).
\end{align}
In order to find the optimal step-sizes and estimate the convergence factors of the iterations, we need to be able to bound the convexity modulus of $d(\mu)$ and bound the Lipschitz constant of its gradient. The following observation is a first step towards this goal:
\begin{lemma} \label{lem:primal_dual}
Consider the optimization problem (\ref{eqn:problem_formulation}) with associated dual function (\ref{eqn:dual_function}). Let $f$ be a continuously differentiable and closed convex function. Then,
\begin{itemize}
\item[(i)] If $f$ is strongly convex with modulus $l$, then $-\nabla d$ is Lipschitz continuous with constant $\lambda_n(AA^\top)
/l$.
\item[(ii)] If $\nabla f$ is Lipschitz continuous with constant $u$, then $-d$ is strongly convex with modulus $\lambda_1(AA^\top)/u$.
\end{itemize}
\end{lemma}
These dual bounds can be used to find step-sizes with strong performance guarantees for the dual iterations. Specifically:
\begin{theorem} \label{thm:dual_convergence}
Consider the smoothness bounds stated in Lemma~\ref{lem:primal_dual}. Then, the accelerated dual iterations~\eqref{eqn:dual_accelerated_iterations} converge to $\mu^\star$ at linear rate with the guaranteed convergence factor
\begin{align*}
q^{\star}=\frac{\sqrt{u \lambda_n(AA^\top)}-\sqrt{l\lambda_1(A A^\top)}}{\sqrt{u\lambda_n(A A^\top)}+\sqrt{l\lambda_1(A A^\top)}},
\end{align*}
obtained for step-sizes:
\begin{align*}
\alpha^\star=\left(\frac{2}{\sqrt{u\lambda_n(AA^\top)}+\sqrt{l\lambda_1(AA^\top)}}\right)^2,\quad  \beta^\star=\left(\frac{\sqrt{u \lambda_n(AA^\top)}-\sqrt{l\lambda_1(AA^\top)}}{\sqrt{u\lambda_n(AA^\top)}+\sqrt{l\lambda_1(AA^\top)}}\right)^2.
\end{align*}
\end{theorem}
The advantage of Theorem~\ref{thm:dual_convergence} is that it provides step-size parameters with guaranteed convergence factor using readily available data of the primal problem. How close to optimal these results are depends on how tight the bounds in Lemma~\ref{lem:primal_dual} are. If the bounds are tight, then the step-sizes in Theorem~\ref{thm:dual_convergence} are truly optimal. The next example shows that a certain degree of conservatism may be present, even for quadratic programming problems.
\begin{example}
Consider the quadratic minimization problem
\begin{align*}
	\begin{array}[c]{ll}
	\mbox{minimize} & \frac{1}{2}x^{\top}Qx\\
	\mbox{subject to} & Ax=b
	\end{array}
\end{align*}
where $Q \in {\mathcal S}^{n}_+$, nonsingular $A\in \mathbb{R}^{n\times n}$ and $b\in \mathbf{R}^n$.  This implies that the objective function is strongly convex with modulus $\lambda_1(Q)$ and that its gradient is Lipschitz-continuous with constant $\lambda_n(Q)$. Hence, according to Lemma~\ref{lem:primal_dual}, $-d$ is strongly convex with modulus $\lambda_1(AA^{\top})/\lambda_n(Q)$ and its gradient is Lipschitz continuous with constant $\lambda_n(AA^{\top})/\lambda_1(Q)$. However, direct calculations reveal that
\begin{align*}
	d(\mu) &= -\frac{1}{2}\mu^{\top}AQ^{-1}A^{\top}\mu -\mu^{\top}b
\end{align*}
from which we see that $-d$ has convexity modulus $\lambda_1(AQ^{-1}A^{\top})$ and that its gradient is Lipschitz continuous with constant $\lambda_n(AQ^{-1}A^{\top})$. By~\cite[p. 225]{HoJ:85}, these bounds are tighter than those offered by Lemma~\ref{lem:primal_dual}. Specifically, for congruent matrices $Q^{-1}$ and $AQ^{-1}A^{\top}$ there exists nonnegative real numbers $\theta_k$ such that $\lambda_1(AA^{\top})\leq \theta_k\leq \lambda_n(AA^{\top})$ and $\theta_k\lambda_k(Q^{-1})=\lambda_k(AQ^{-1}A^{\top})$. For $k=1$ and $n$ we obtain
\begin{align*}
\frac{\lambda_1(AA^{\top})}{\lambda_n(Q)} &\leq \lambda_1(AQ^{-1}A^{\top}), \quad
\lambda_n(AQ^{-1}A^{\top})\leq \frac{\lambda_n(AA^{\top})}{\lambda_1(Q)}
\end{align*}
For some important classes of problems, the bounds are, however, tight. One such example is the average consensus application considered in Section~\ref{sec:applications}.
\end{example}

\section{Robustness analysis} \label{sec:robustness}
The proposed multi-step methods have significantly improved convergence factors compared to the gradient iterations, and particularly so when the Hessian of the loss function and/or the graph Laplacian of the network is ill-conditioned. However, to design the optimal step-sizes for the multi-step methods one needs to know the upper and lower bounds on the Hessian and the largest and smallest non-zero eigenvalue of the graph Laplacian. These quantities can be hard to estimate accurately in practice. It is therefore important to analyze the sensitivity of the multi-step methods to errors in these parameters to assess if the performance benefits prevail when the step-sizes are tuned based on misestimated parameters. Such a robustness analysis will be performed next.

Let $\uest{\lambda}$ and $\oest{\lambda}$ denote the estimates of $\underline{\lambda}$ and $\overline{\lambda}$ available when tuning the step-sizes. We are interested in quantifying how the convergence properties, and the convergence factors, of the gradient and the multi-step methods are affected when $\uest{\lambda}$ and $\oest{\lambda}$ are used in the step-size formulas that we have derived earlier.  Theorem~\ref{thm:multi_step} provides some useful observations for the multi-step method. The corresponding results for the weighted gradient method are summarized in the following lemma:
\begin{lemma}\label{lem:weighted_gradient_convergence}
Consider the weighted gradient iterations (\ref{eqn:scaled_gradient_iteration}) and let $\overline{\lambda}$ and $\underline{\lambda}$ denote the largest and smallest non-zero eigenvalue of $WH$, respectively. Then,  for fixed step-size $0<\alpha<2/\overline{\lambda}$ (\ref{eqn:scaled_gradient_iteration}) converges to  $x^{\star}$ at linear rate with convergence factor
\begin{align*}
	q_G &= \max\left\{ \vert 1-\alpha \underline{\lambda}\vert, \; \vert 1-\alpha \overline{\lambda}\vert \right\}
\end{align*}
The minimal value $q_G^{\star}=(\overline{\lambda}-\underline{\lambda})/(\overline{\lambda}+\underline{\lambda})$ is obtained for the step-size $\alpha=2/(\overline{\lambda}+\underline{\lambda})$.
\end{lemma}

Combining this lemma with our previous results from Theorem~\ref{thm:multi_step} yields the following observation.
\begin{proposition}\label{prop:robustness_stability_conditions}
Let $\uest{\lambda}$ and $\oest{\lambda}$ be estimates of $\underline{\lambda}$ and $\overline{\lambda}$, respectively, and assume that $0<\uest{\lambda}<\oest{\lambda}$. Then, for all values of $\uest{\lambda}$ and $\oest{\lambda}$ such that $\overline{\lambda} < \oest{\lambda}+\uest{\lambda}$, both the weighted gradient iteration (\ref{eqn:scaled_gradient_iteration}) with step-size
\begin{align}
	\oest{\alpha}=2/(\oest{\lambda}+\uest{\lambda}) \label{eqn:gradient_step_inaccurate}
\end{align}
and the multi-step method variant (\ref{eqn:multistep_scaled_gradient}) with
\begin{align}
	\oest{\alpha}&=\left(
\frac{2}{\sqrt{\oest{\lambda}}+\sqrt{\uest{\lambda}}}
\right)^2, \; \oest{\beta}=\left(\frac{\sqrt{\oest{\lambda}}-\sqrt{\uest{\lambda}}}{\sqrt{\oest{\lambda}}+\sqrt{\uest{\lambda}}} \right)^2 \label{eqn:multistep_step_inaccurate}
\end{align}
converge to the optimizer $x^{\star}$ of (\ref{eqn:problem_formulation}).
\end{proposition}
In practice, one should expect that $\oest{\lambda}$ is overestimated, in which case both methods converge. However, convergence can be guaranteed for a much wider range of perturbations. Figure~\ref{fig:robust_1} considers perturbations of the form $\uest{\lambda}=\underline{\lambda}+\uest{\varepsilon}$ and $\oest{\lambda}=\overline{\lambda}+\oest{\varepsilon}$. The white area is the locus of perturbations for which convergence is guaranteed, while the dark area represents inadmissible perturbations which render either $\uest{\lambda}$ or $\oest{\lambda}$ negative. Note that both algorithms are robust to a continuous departure from the true values of $\underline{\lambda}$ and $\overline{\lambda}$, since there is a ball with radius $\sqrt{3}\underline{\lambda}/2$  around the true values for which the methods are guaranteed to converge.
\begin{figure}[tb]
    \begin{center}
     \subfigure[Stability regions]{
     \label{fig:robustness_1_subfig1}
    \includegraphics[width=0.46 \columnwidth]{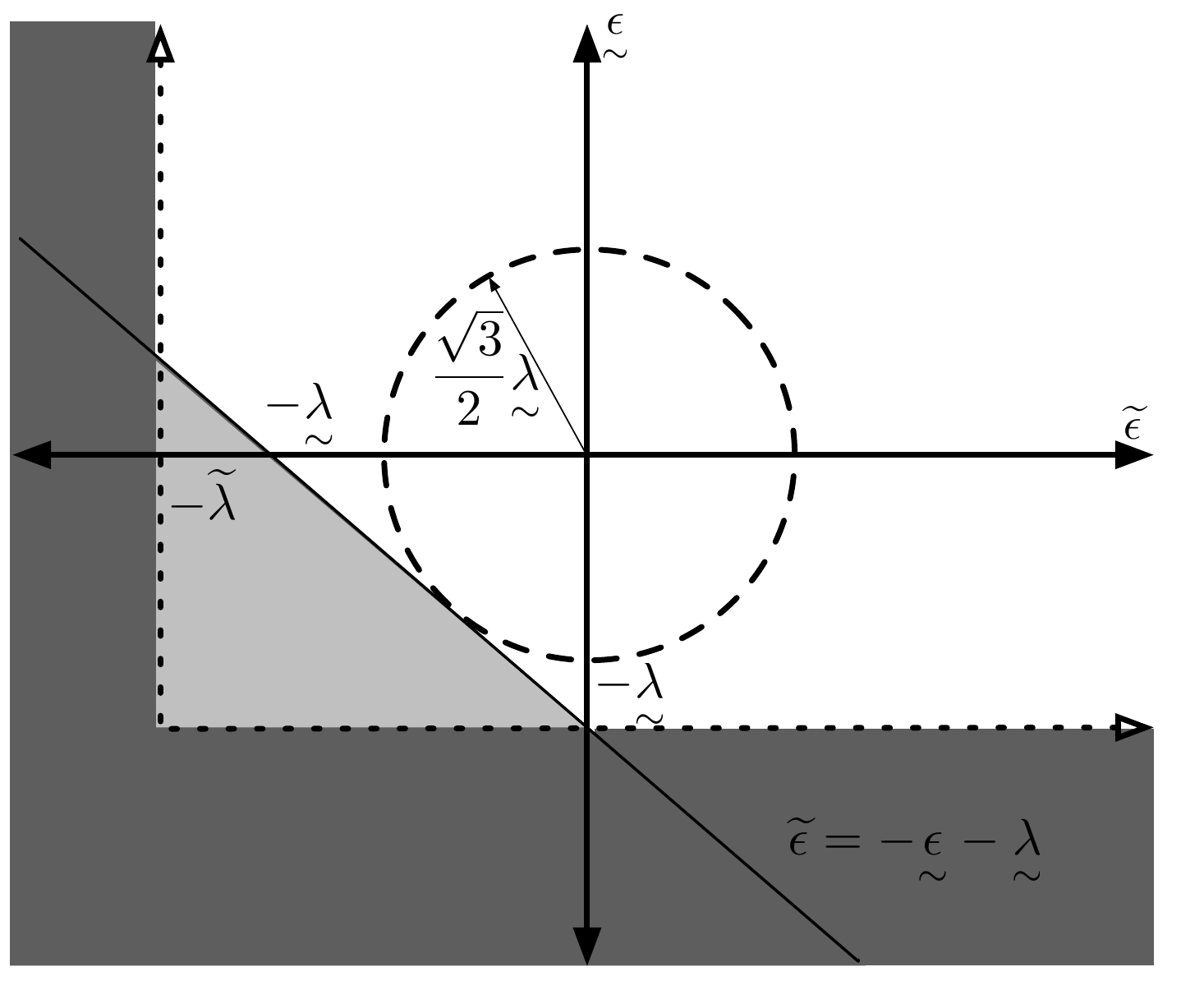}}
      \subfigure[Different perturbation regions]{
      \label{fig:robustness_1_subfig2}
    \includegraphics[width=0.46 \columnwidth]{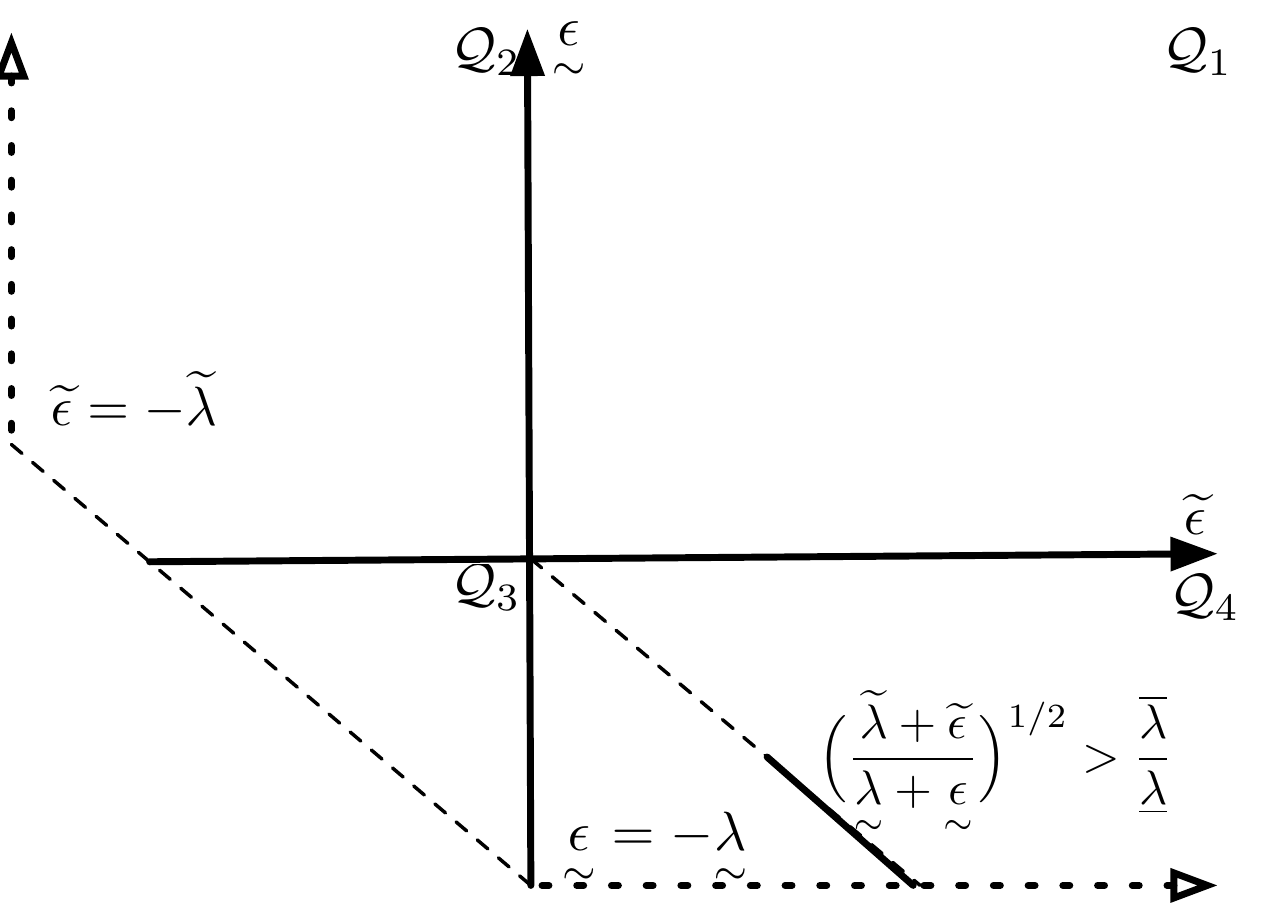}}
    \end{center}
    \vspace{-2mm}
  \caption{Perturbations in the white and gray area correspond to the stable and unstable regions of multi-step algorithm respectively. (b) Multi-step algorithm outperforms gradient iterations in $(\underset{\widetilde{}}{\varepsilon}, \oest{\varepsilon})\in {\mathcal C}\backslash {\mathcal Q}_4$. For symmetric errors in ${\mathcal Q}_4$ (along the line $\oest{\varepsilon}=-\underset{\widetilde{}}{\varepsilon}$) gradient might outperform multi-step algorithm. This condition is depicted in the plot as a solid line.}
    \label{fig:robust_1}
\end{figure}

Next, we proceed to compare the convergence \emph{factors} of the two methods when the step-sizes are tuned based on inaccurate parameters. The following Lemma is then useful.
\begin{lemma} \label{lem:conv_factors_inaccurate}
Let $\uest{\lambda}$ and $\oest{\lambda}$ satisfy $0<\overline{\lambda}<\uest{\lambda}+\oest{\lambda}$.  The convergence factor of the weighted gradient method (\ref{eqn:scaled_gradient_iteration}) with step-size (\ref{eqn:gradient_step_inaccurate}) is given by
\begin{align}\label{eqn:gradient_inaccurate_factor}
	\oest{q}_G &= \begin{cases}
	2\overline{\lambda}/(\uest{\lambda}+\oest{\lambda})-1 & \mbox{ if } \uest{\lambda}+\oest{\lambda} \leq \underline{\lambda}+\overline{\lambda}\\
	1-2\underline{\lambda}/(\uest{\lambda}+\oest{\lambda}) & \mbox{otherwise}, 
	\end{cases}
\end{align}
while the multi-step weighted gradient method (\ref{eqn:multistep_scaled_gradient}) with step-sizes (\ref{eqn:multistep_step_inaccurate}) has convergence factor
\begin{align}\label{eqn:multistep_inaccurate_factor}
	\oest{q} &= \max \left\{ \sqrt{\widetilde{\beta}}, \vert 1+\oest{\beta}-\oest{\alpha}\underline{\lambda}\vert-\sqrt{\widetilde{\beta}}, \vert 1+\oest{\beta}-\oest{\alpha}\overline{\lambda}\vert-\sqrt{\oest{\beta}}\right\}
\end{align}
\end{lemma}

The convergence factor expressions derived in Lemma~\ref{lem:conv_factors_inaccurate} allow us to come to the following conclusions:
\begin{proposition} \label{prop:conv_factors_inaccurate}
Let $\uest{\lambda}=\underline{\lambda}+\uest{\varepsilon}$, $\oest{\lambda}=\overline{\lambda}+\oest{\varepsilon}$ and define the set of perturbation under which the methods converge
\begin{align*}
{\mathcal C} &= \{ (\uest{\varepsilon}, \oest{\varepsilon})\;\vert\; \uest{\varepsilon}\geq -\underline{\lambda}, \, \oest{\varepsilon}\geq -\overline{\lambda},\, \uest{\varepsilon}+\oest{\varepsilon}\geq -\underline{\lambda}\}
\end{align*}
and  the fourth quadrant in the perturbation space ${\mathcal Q}_4 = \{ (\uest{\varepsilon}, \oest{\varepsilon}) \;\vert\; \uest{\varepsilon} <0 \,\cap\, \oest{\varepsilon}>0 \}$. Then, for all $(\uest{\varepsilon}, \oest{\varepsilon})\in {\mathcal C}\backslash {\mathcal Q}_4$, it holds that $\oest{q}\leq \oest{q}_G$.
However, there exists $(\uest{\varepsilon}, \oest{\varepsilon})\in {\mathcal Q}_4$ for which the scaled gradient has a smaller convergence factor than the multi-step variant. In particular, for
\begin{align}
\label{eqn:robustness_gradient_suprier}
	(\uest{\varepsilon}, \oest{\varepsilon})\in {\mathcal Q}_4 \,\,\mbox{and}\,\,(\overline{\lambda}+\oest{\varepsilon})/(\underline{\lambda}+\uest{\varepsilon}) &\geq  (\overline{\lambda}/\underline{\lambda})^2
\end{align}
the multi-step iterations (\ref{eqn:multistep_scaled_gradient}) converge slower than (\ref{eqn:scaled_gradient_iteration}).
\end{proposition}

\begin{figure}[tb]
    \begin{center}
     \subfigure[Symmetric perturbations in ${\mathcal Q}_4$]{
     \label{fig:robustness_symmetric}
    \includegraphics[width=0.48 \columnwidth]{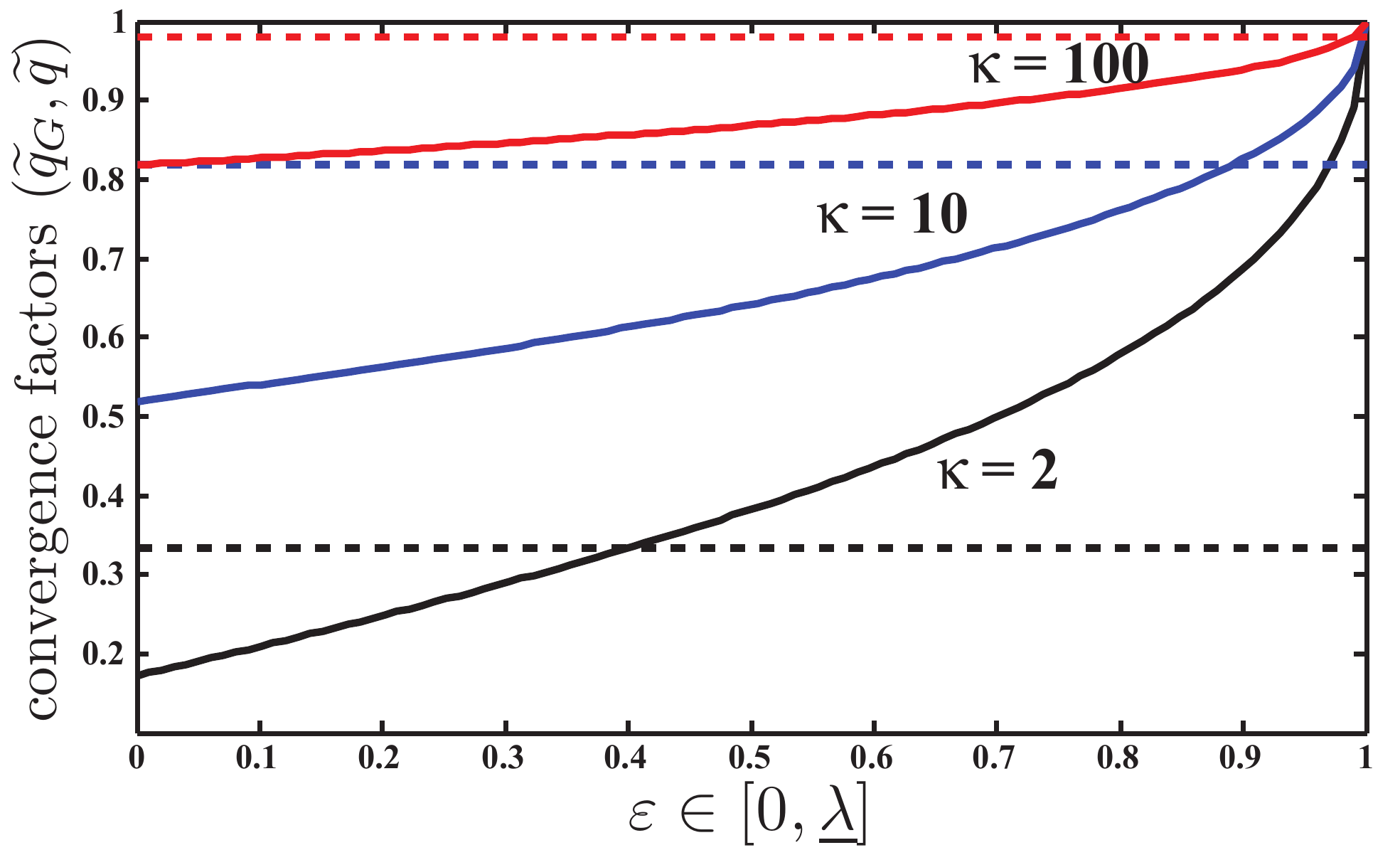}}
      \subfigure[General perturbation in ${\mathcal Q}_4$]{
      \label{fig:robustness_contour}
    \includegraphics[width=0.48 \columnwidth]{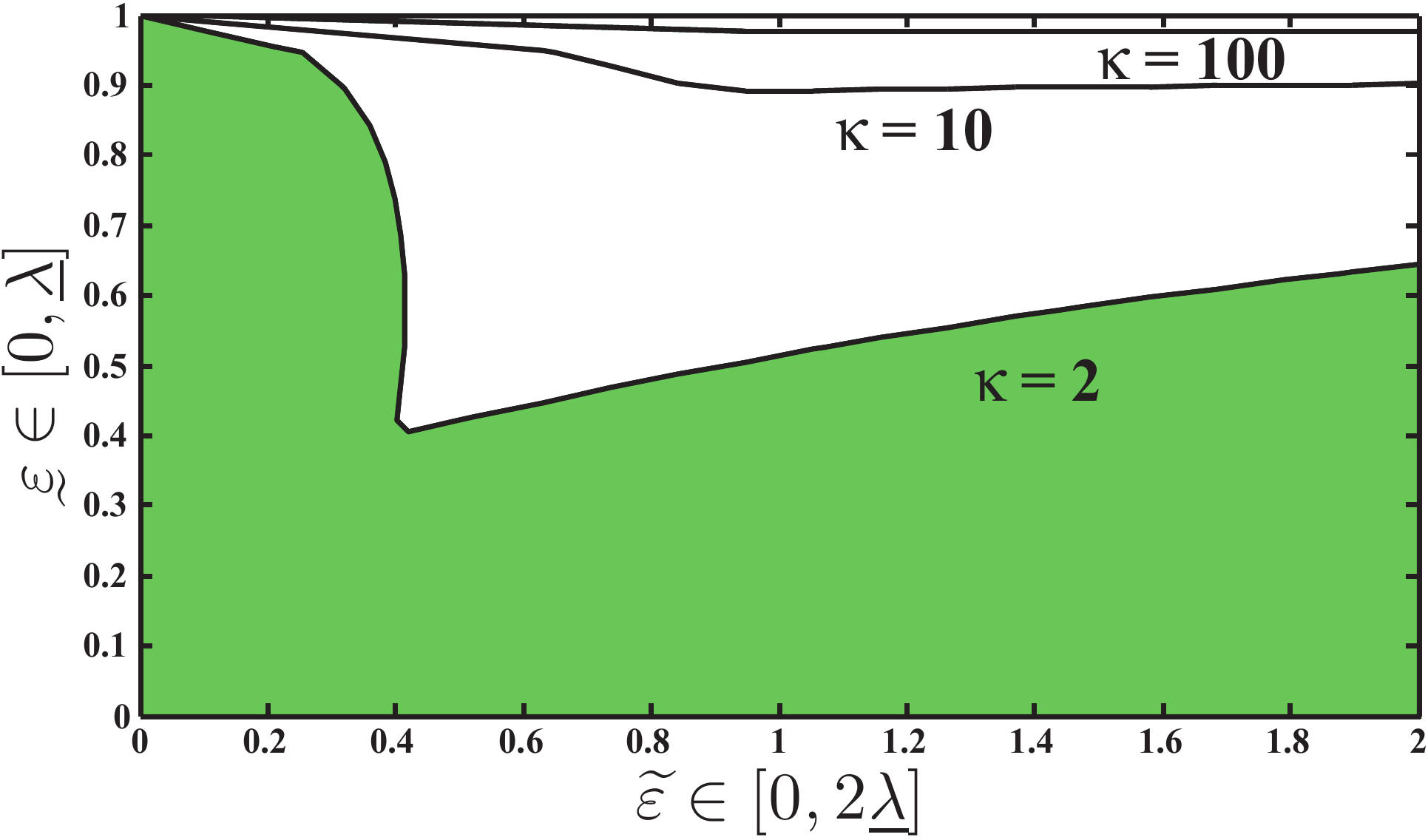}}
    \end{center}
    \vspace{-2mm}
   \caption{(a) Convergence factor of multi-step and gradient algorithms under the condition described by~\eqref{eqn:robustness_gradient_suprier}. Solid lines belong to $\oest{q}$ while the dashed lines depict $\oest{q}_{G}$. (b) Level curves of $\oest{q}-\oest{q}_{G}$ around the origin for~$(\underset{\widetilde{}}{\varepsilon}, \oest{\varepsilon})\in {\mathcal Q}_4$.}
    \label{fig:robust_Q4}
\end{figure}

Fig.~\ref{fig:robustness_1_subfig2} illustrates the different perturbations considered in Proposition~\ref{prop:conv_factors_inaccurate}. While the  multi-step method has superior convergence rate for most perturbations, the troublesome region ${\mathcal Q}_4$ is envisaged to be the most likely one in engineering applications. Because it represents the perturbations where the smallest eigenvalue is underestimated while the largest eigenvalue is overestimated. To shed more light on the convergence properties in this region, we perform a numerical study on a quadratic function with $\underline{\lambda}=1$ and $\overline{\lambda}$ varying from $2$ to $100$. We first consider symmetric perturbations $\uest{\varepsilon}=-\oest{\varepsilon}$, in which case the convergence factor of the gradient method is $\oest{q}_G=  1-2/(1+\overline{\lambda}/\underline{\lambda})$ while the convergence factor of the multi-step method is $\oest{q}=1-2/\sqrt{1+\oest{\lambda}/\uest{\lambda}}$. Fig.~\ref{fig:robustness_symmetric} shows the convergence factors as a function of the perturbation $\varepsilon=\oest{\varepsilon}$. The convergence factor of the gradient iterations is insensitive to this class of perturbations, while the performance of the multi-step iterations degrades with the size of the perturbation, and will eventually become inferior to the gradient. To complement this analysis, we also sweep over $(\uest{\varepsilon}, \oest{\varepsilon})\in {\mathcal C}\cap {\mathcal Q}_4$ and compute the convergence factors for the two methods for problems with different $\overline{\lambda}$. The plot in Fig.~\ref{fig:robustness_contour} indicates that when the condition number $\overline{\lambda}/\underline{\lambda}$ increases, the area where the gradient method is superior (the area above the contour line) is shrinking. It also shows that when $\uest{\lambda}$ tends to zero or $\oest{\lambda}$ is very large, the performance of the multi-step method is severely degraded.

\section{Applications} \label{sec:applications}
In this section, we will apply the developed techniques to three classes of engineering problem for which distributed optimization techniques have received significant attention. These are resource-allocation subject to a network-wide resource-constraint, distributed averaging consensus, and Internet congestion control. In all cases, we will demonstrate that significant speed-ups can be achieved by direct applications of our results, even when compared to acceleration techniques that have been tailor-made to the specific problem class.

\subsection{Accelerated resource allocation}

Our first application is the distributed resource allocation problem under a network-wide resource constraint described in Example~\ref{ex:total_budget_constraint}. This problem class was introduced in~\cite{HSS:80} and revisited by~\cite{XiB:06}, who developed optimal and heuristic weights for the corresponding weighted gradient iteration (\ref{eqn:scaled_gradient_iteration}).
We hence compare the multi-step method developed in this paper with the optimal and suboptimal tuning for the standard weighted gradient iterations proposed in ~\cite{XiB:06}.
Similarly to~\cite{XiB:06} we create problem instances by generating random networks and assigning loss functions on the form $f_v(x_v)=a_v(x_v-c_v)^2+\log[1+\mbox{exp}(x_v-d_v)]$ to nodes. The parameters $a_v, b_v, c_v$ and $d_v$ are drawn uniformly from the intervals $[0,2]$, $[-2,2]$, $[-10,10$ and $[-10, 10]$, respectively. In~\cite{XiB:06} it was shown that the second derivatives of these functions are bounded by $l_v=a_v$ and $u_v=a_v+b_v^2/4$.

\begin{figure}[tb]
 \centering
\includegraphics[angle=0,width=.6\columnwidth]{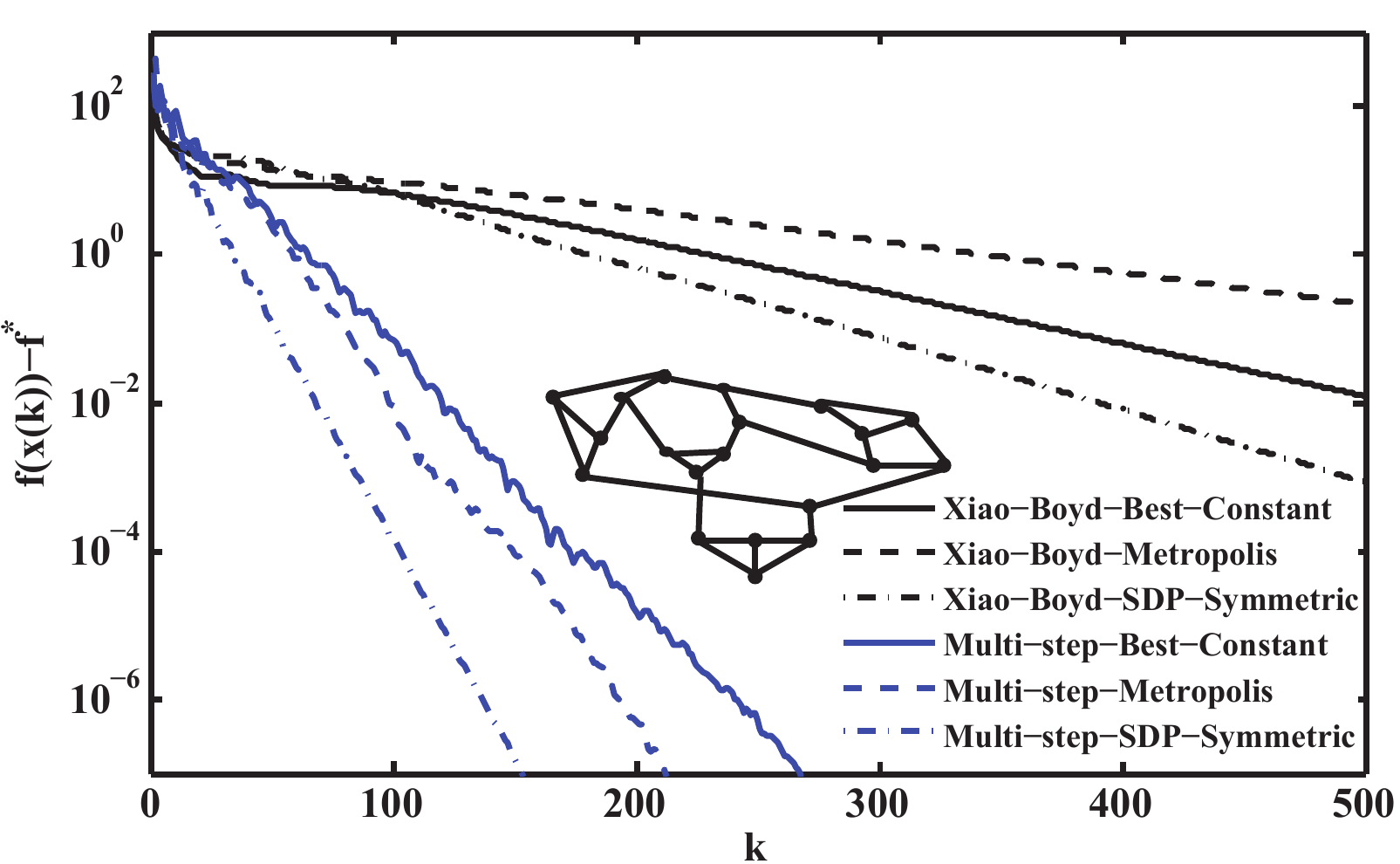}
\caption{ Convergence behavior convergence behavior for weighted and multi-step weighted gradient iterations using randomly generated network and the heuristic weights. plot shows $f\big(x(k)\big)-f^{\star}$ versus  iteration number $k$.}
\label{fig:resorce}
\end{figure}

Fig.~\ref{fig:resorce} shows a representative example of a problem instance along with the convergence behavior for weighted and multi-step weighted gradient iterations for several weight choices. The optimal weights for the weighted gradient method can be found by solving a semi-definite program derived in~\cite{XiB:06}, and by Proposition~\ref{prop:weight_optimization} for the multi-step variant.  In addition, we use the heuristic weights ``best constant'' and ``metropolis'' introduced in~\cite{XiB:06}. In all cases, we observe significantly improved convergence factors for the multi-step method.

In addition to simulations, we compare the analytical expressions for the convergence factors of the weighted gradient and multi-step iterations.
Table~\ref{tab:conv_factors} again demonstrates superior performance of the multi-step method. In addition to the heuristic weights considered previously, we have also used the ``max-degree'' weight heuristic from~\cite{XiB:06}. While this weight setting tends to be worse than ``best constant'' for the scaled gradient iterations, the two methods will always result in the same convergence factors for the multi-step method. This follows from Remark~\ref{rem:weight_scaling} and the fact that both heuristics generate weight matrices on the form $\gamma {\mathcal L}$ where ${\mathcal L}$ is the Laplacian of the underlying graph and $\gamma$ is a positive scalar.
\vspace{-3mm}
\begin{table}[tb]
\renewcommand{\arraystretch}{1}
\caption{resource allocation: Guaranteed convergence factors}
\vspace{-3mm}
\centering
\label{tab:conv_factors}
\begin{tabular}{c c c c c}
\hline
\bfseries Method & \bfseries Max-degree & \bfseries Metropolis & \bfseries Best Constant & \bfseries SDP \\
\hline
Xiao-Boyd &   0.9420 &  0.9318&   0.9133&  0.8952\\
\hline
Multi-step &   0.8667 &   0.8565&   0.8667& 0.7604\\
\hline
\end{tabular}
\vspace{-2mm}
\end{table}
\subsection{Distributed averaging and consensus}

\begin{figure}[tb]
\centering
\includegraphics[angle=0,width=.6\columnwidth]{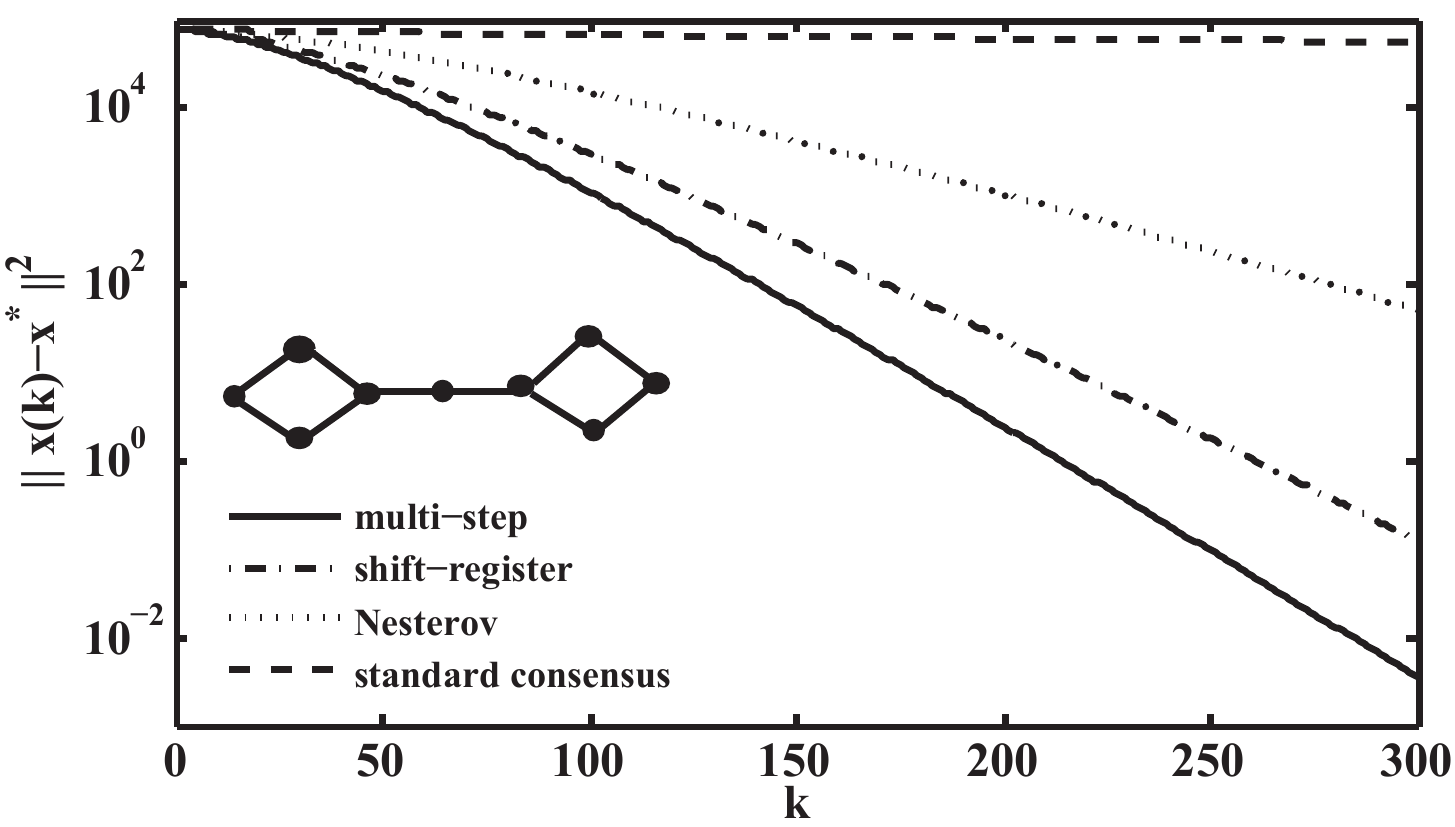}
\caption{ Comparison of standard, multi-step, shift-register, and Nesterov consensus algorithms using metropolis wights. simulation on a dumbbell of $100$ nodes: log scale of objective function  $\|x(k)-x^{\star}\|_2^2$ versus  iteration number k. algorithms start from common initial point $x(0)$.}
\label{fig:consensus}
\end{figure}
Our second application is devoted to distributed averaging.  Distributed algorithms for consensus seeking have been researched intensively for decades, see \emph{e.g.}~\cite{Tsitsiklis86,jadbabaie03,rabbat010}. Here, each node $v$ in the network initially holds a value $c_v$ and coordinates with neighbors in the graph to find the network-wide average. Clearly, this average can be found by applying any distributed optimization technique to the problem
\begin{align}
	\begin{array}[c]{ll}
		\underset{x}{\mbox{minimize}} & \sum_{v\in {\mathcal V}} \frac{1}{2}(x-c_v)^2
	\end{array} \label{eqn:consensus_ls}
\end{align}
since the optimal solution to this problem is the network-wide average of the constants $c_v$. In particular, we will explore how the multi-step technique described in Example~\ref{ex:consensus} with our optimal parameter selection rule compares with the state-of-the art distributed averaging algorithms from the literature.

The basic consensus algorithms use iterations on the form
\begin{align}
	x_v(k+1) &= Q_{vv}x_v(k)+\sum_{w\in {\mathcal N}_v} Q_{vw}x_{w}(k), x\label{eqn:basic_consensus}
\end{align}
where $Q_{vw}$ are scalar weights, and the node states are initialized with $x_v(0)=c_v$. The paper~\cite{XiB:04} provides necessary and sufficient conditions on the weight matrix $Q=[Q_{vw}]$ for the iterations to converge to the network-wide average of the initial values, along with computational procedures for finding $Q$ that minimizes the convergence factor of the iterations.

Following the steps of Example~\ref{ex:consensus}, the optimization approach to consensus would suggest the iterations
\begin{align}
	x(k+1) &= x(k)-\alpha W x(k) \label{eqn:optimization_based_consensus}
\end{align}
with $W=A^{\top}A$ where $A$ is the incidence matrix of ${\mathcal G}$. These iterations are on the same form as (\ref{eqn:basic_consensus}) but use a particular weight matrix. The multi-step counterpart of (\ref{eqn:optimization_based_consensus}) is
\begin{align}
	x(k+1) &= \left( (1+\beta)I-\alpha W \right)x(k)-\beta x(k-1) \label{eqn:optimization_based_consensus_multistep}
\end{align}
In a fair comparison between the multi-step iterations (\ref{eqn:optimization_based_consensus_multistep}) and the basic consensus iterations, the weight matrices of the two approaches should not necessarily be the same, nor necessarily equal to the graph Laplacian. Rather, the weight matrix for the  consensus iterations (\ref{eqn:basic_consensus}) should be optimized using the results from~\cite{XiB:04} and the weigh matrix for the multi-step iteration should be  computed using Proposition~\ref{prop:weight_optimization}.

In addition to the basic consensus iterations with optimal weights, we will also compare our multi-step iterations with two alternative acceleration schemes from the literature. The first one comes from the literature on accelerated consensus and uses shift registers~\cite{Cao06},~\cite{young72},~\cite{anderson09}. Similarly to the multi-step method, these techniques use a history of past iterates, stored in local registers, when computing the next. For the basic consensus iterations (\ref{eqn:basic_consensus}), the shift register yields
\begin{align}
	x(k+1) &= \zeta Q x(k) + (1-\zeta) x(k-1) \label{eqn:shift_register_consensus}
\end{align}
The current approaches to consensus based on shift-registers assume that $Q$ is given and design $\zeta$ to minimize the convergence factor of the iterations. The key results can be traced back to Golub and Varga~\cite{GoV:61} who determined the optimal $\zeta$ and the associated convergence factor to be
\begin{align}\label{eqn:shift_register_optimal_parameters}
	\zeta^{\star} = \frac{2}{1+\sqrt{1-\lambda_{n-1}^2(Q)}},
	\quad
	q^{\star}_{SR} = \sqrt{
	\frac{1-\sqrt{1-\lambda_{n-1}^2(Q)}}
	{1+\sqrt{1-\lambda_{n-1}^2(Q)}}
	}
\end{align}
In our comparisons, the shift-register iterations will use the $Q$-matrix optimized for the basic consensus iterations and the associated $\zeta^{\star}$ given above.  The second accleration technique that we will compare with is the order-optimal gradient methods developed by Nesterov~\cite{Nesterov04}. While these techniques have optimal convergence \emph{rate}, also in the absence of strong convexity, they are not guaranteed to obtain the best convergence factors. For the case of an objective function which is strongly convex with modulus $l$ and whose gradient is Lipschitz continuous with constant $u$, the following iterations are proposed in~\cite{Nesterov04}:
\begin{align*}
	\hat{x}(k+1) &= x(k)-\nabla f(x(k))/u\\
	x(k+1) &= \hat{x}(k+1)+\frac{\sqrt{u}-\sqrt{l}}{\sqrt{u}+\sqrt{l}}(\hat{x}(k+1)-\hat{x}(k))
\end{align*}
initialized with $\hat{x}(0)=x(0)$. When we apply this technique to the consensus problem, we arrive at the iterations
\begin{align}
	x(k+1) &= (I-\alpha W)\left( x(k)+ b( x(k)-x(k-1) )\right) \label{eqn:nesterov_consensus}
\end{align}
with parameters $W=AA^{\top}, a=\lambda_n^{-1}(W)$ and
$b=(\sqrt {\lambda_n(W)}-\sqrt{\lambda_2(W)})/(\sqrt{\lambda_n(W)}+\sqrt{\lambda_2(W)})$.

Fig.~\ref{fig:consensus} compares the multi-step iterations (\ref{eqn:optimization_based_consensus_multistep}) developed in this paper with (a) the basic consensus iterations (\ref{eqn:basic_consensus}) with a weight matrix determined using the metropolis scheme, (b) the shift-register acceleration (\ref{eqn:shift_register_consensus}) with the same weight matrix and the optimal $\zeta$, and (c) the order-optimal method (\ref{eqn:nesterov_consensus}). The particular results shown are for a network of $100$ nodes in a dumbbell topology. The simulations show that all three methods yield a significant improvement in convergence factors over the basic iterations, and that the multi-step method developed in this paper outperforms the alternatives.

Several remarks are in order. First, since the Hessian of (\ref{eqn:consensus_ls}) equals the identity matrix, the speed-up of the multi-step iterations are proportional to $\sqrt{\kappa}=\sqrt{\lambda_n(W)/\lambda_{2}(W)}$. When $W$ equals ${\mathcal L}$, the Laplacian of the underlying graph, we can quantify the speed-ups for certain classes of graphs using spectral graph theory~\cite{Chu:97}. For example, the complete graph has $\lambda_2({
\mathcal L})=\lambda_n({\mathcal L})$ so $\kappa=1$ and there is no real advantage of the multi-step iterations. On the other hand, for a ring network the eigenvalues of ${\mathcal L}$ are given by $1-\cos(2\pi v)/\vert {\mathcal V}\vert$, so $\kappa$ grows quickly with the number of nodes, and the performance improvements of \ref{eqn:optimization_based_consensus_multistep}) over (\ref{eqn:optimization_based_consensus})  could be substantial.

Our second remark pertains to the shift-register iterations. Since these iterations have the same form as (\ref{eqn:optimization_based_consensus_multistep}), we can go beyond the current literature on shift-register consensus (which assumes $Q$ to be given and optimizes $\zeta$) and provide jointly optimal weight matrix and $\zeta$-parameter:

\begin{proposition}\label{prop:shift_register_optimal_parameters}
The weight matrix $Q^\star$ and constant $\zeta^\star$ that minimizes the convergence factor of the shift-register consensus iterations (\ref{eqn:shift_register_consensus}) are
\begin{align*}
	Q^\star &= I-\theta^\star W^{\star}, \quad \zeta^\star=1+\beta^{\star}
\end{align*}
where $W^{\star}$ is computed in Proposition~\ref{prop:weight_optimization}, $\beta^{\star}$ is given in Theorem~\ref{thm:multi_step} with $H=I$ and
\begin{align*}
	\theta^\star &= \frac{2}{\lambda_2(W^\star)+\lambda_n(W^\star)}
\end{align*}
\end{proposition}

\subsection{Internet congestion control}

Our final application is to the area of Internet congestion control, where Network Utility Maximization (NUM) has emerged as powerful framework
for studying various important resource allocation problems, see, \emph{e.g.}, \cite{kelly98,low99,Mikael04,Chaing07}. The vast majority of the work in this area is based on the dual decomposition approach introduced in~\cite{low99}. Here, the optimal bandwidth sharing among $S$ flows in a data network is posed as the optimizer of a convex optimization problem
\begin{align}
	\begin{array}[c]{ll}
	\underset{x}{\mbox{minimize}} & \sum_{s} u_s(x_s)\\
	\mbox{subject to} & x_s\in [m_s, M_s]\\
	& Rx\leq c
	\end{array}	 \label{eqn:prob_num}
\end{align}
In this formulation $x_s$ is the communication rate of flow $s$, and the strictly concave and increasing function $u_s(x_s)$ describes the utility that source $s$ has of communicating at rate $x_s$. The communication rate is restricted to a bounded interval. Finally, $R\in \{0, 1\}^{L\times S}$ is a routing matrix, whose entries $R_{ls}$ equal one if flow $s$ traverses link $l$ and zero otherwise. In this way, $Rx$ is the total traffic on links, which cannot exceed the link capacities $c\in {\mathbb R}^n$. We make the following assumptions.
\begin{assumption}\label{ass:num}
For the problem~(\ref{eqn:prob_num}) it holds that
\begin{itemize}
\item[(i)] Each $u_s(x_s)$ is twice continuously differentiable and satisfies $0<l<-\nabla^2 u_s(x_s)<u$ for $x_s\in [m_s,M_s]$
\item[(ii)] For every link $l$, there exists a source $s$ whose flow only traverses $l$, \emph{i.e.} $R_{ls}=1$ and $R_{l^{\prime}s}=0$ for all $l^{\prime}\neq l$.
\end{itemize}
\end{assumption}
While these assumptions appear restrictive, they are often postulated in the literature (\emph{e.g.}~\cite[Assumptions C1-C4]{low99}). Note that under Assumption~\ref{ass:num}, the routing matrix has full row rank and all link constraints hold with equality at optimum. Hence, we can replace $Rx\leq c$ in (\ref{eqn:prob_num}) with $Rx=c$, introduce Lagrange multipliers $\mu$ for these constraint,  and form the associated dual function
\begin{align*}
	d(\mu) &= \max_{x_{s}\in [m_s, M_s]} \sum_s \left\{ u_s(x_s)-x_s\sum_l R_{ls}\mu_l \right\}+\sum_l \mu_l c_l
\end{align*}
Evaluating $d(\mu)$ amounts to solving an optimization problem in $x$. Since this problem is separable in $x_s$, it can be solved by each source in isolation based on the sum of the Lagrange multipliers for the links that the flow traverses,
\begin{align}
	x_s^{\star}(\mu) &= \underset{z \in [m_s, M_s]}{\arg\max} u_s(z)-z\sum_{l} R_{ls}\mu_l  \label{eqn:num_primal}
\end{align}
Similarly, each Lagrange multiplier update
\begin{align}
	\mu_l(k+1) &= \mu_l(k)+\alpha\left(\sum_{l} R_{ls} x_s^{\star}(\mu(k))-c_l \right) \label{eqn:num_dual}
\end{align}
can be updated by the corresponding link based on local information: if the traffic demand on the link exceeds capacity, the multiplier is increased, otherwise it is decreased. It is possible to show that under the conditions that under Assumption~\ref{ass:num},  the dual function is strongly concave, differentiable and has a Lipschitz-continuous gradient~\cite{low99}. Hence, by standard arguments, the updates (\ref{eqn:num_primal}), (\ref{eqn:num_dual}) converge to a primal-dual optimal point $(x^{\star}, \mu^{\star})$ for appropriately chosen step-size $\alpha$.

Our results from Section~\ref{sec:multistep_dual} indicate that substantially improved convergence factors could be obtained by the following class of multi-step updates of the Lagrange multipliers
\begin{align}
	\mu_l(k+1) &= \mu_l(k)+\alpha\left(\sum_{l} R_{ls} x_s^{\star}(\mu(k))-c_l \right) +\beta (\mu_l(k)-\mu_l(k-1)) \label{eqn:num_dual_multistep}
\end{align}
To tune the step-sizes in an optimal way, we bring the techniques from Section~\ref{sec:multistep_dual} into action. To do so, we first bound the eigenvalues of  $RR^{\top}$ using the following result:
\begin{lemma} \label{lem:eigenvalue_bouds}
Let $R\in\{0,1\}^{L\times S}$ satisfy Assumption~\ref{ass:num}. Then
\begin{align*}
1&\leq \lambda_1(RR^{\top}), \quad \lambda_n(RR^{\top})\leq l_{\max}s_{\max}
\end{align*}
where $l_{\max}=\max_s \sum_l R_{ls}$ and $s_{\max}=\max_l \sum_s R_{ls}$.
\end{lemma}
The optimal step-size parameters and corresponding convergence factor now follow from Lemma~\ref{lem:eigenvalue_bouds} and Theorem~\ref{thm:dual_convergence}:
\begin{proposition}
Consider the network utility maximization problem (\ref{eqn:prob_num}) under Assumption~\ref{ass:num}. Then, for $0\leq \beta<1$ and
$0<\alpha<2(1+\beta)/(u l_{\max}s_{\max})$ the iterations (\ref{eqn:num_primal}) and (\ref{eqn:num_dual_multistep}) converge linearly to a primal-dual optimal pair. The step-sizes
\begin{align*}
	\alpha &= \left( \frac{2}{\sqrt{ul_{\max}s_{\max}}+\sqrt{l}}\right)^2, \; \beta=\left( \frac{\sqrt{u l_{\max} s_{\max}}-\sqrt{l}}{\sqrt{ul_{\max}s_{\max}}+\sqrt{l}}\right)^2
\end{align*}
ensure that the convergence factor of the dual iterates is
\begin{align*}
	q_{\rm NUM} &= \frac{\sqrt{u l_{\max}s_{\max}}-\sqrt{l}}{\sqrt{u l_{\max}s_{\max}}+\sqrt{l}}
\end{align*}
\end{proposition}
\begin{figure}[tb]
\centering
\includegraphics[angle=0, width=.6\columnwidth]{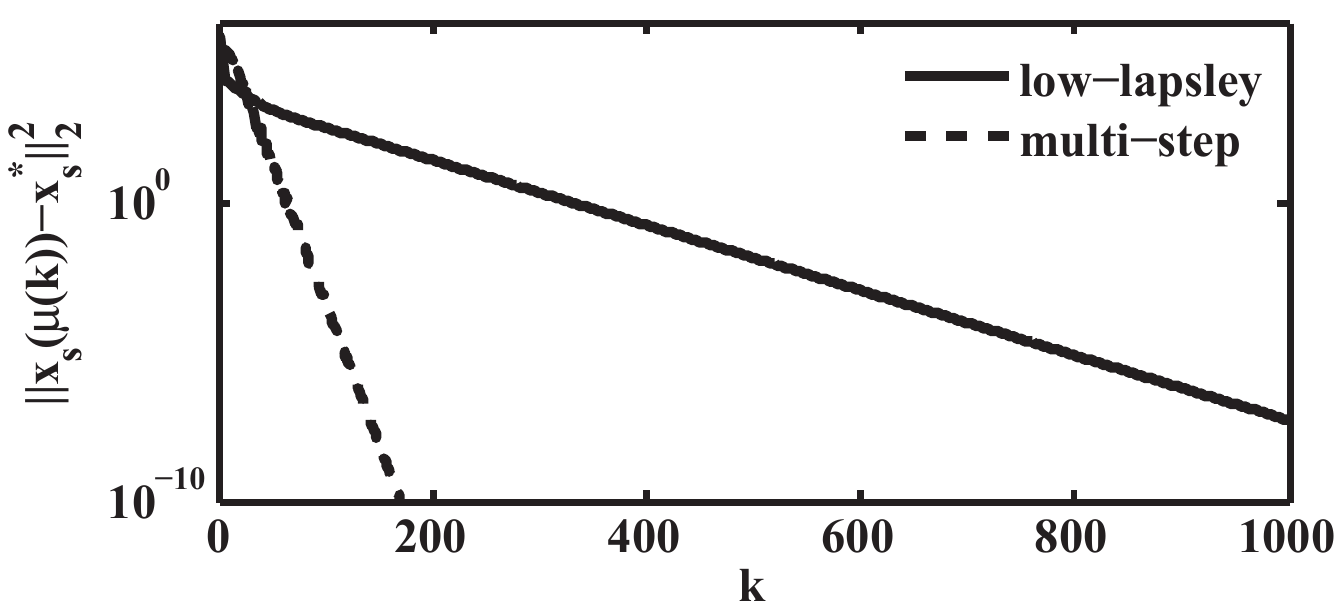}
\caption{ Convergence of Low-Lapsley versus multi-step formulations. Plot shows log scale of the euclidian distance from optimal source rates $\|x_s(\mu(k))-x_s^{\star}\|_2^2$ versus the iteration number $k$.}
\label{fig:num_acceleration}
\end{figure}

Note that an upper bound of the Hessian of the dual function was also derived in~\cite{low99}. However,  strong concavity was not explored and the associated bounds were not derived.

We now comment on the steady behavior of accelerated link price algorithm~\eqref{eqn:num_dual_multistep}. Due to the saturation assumption as $k\rightarrow \infty$, close to the equilibrium, we have $ \alpha\left(\sum_{l} R_{ls} x_s^{\star}(\mu(k))-c_l \right)  \rightarrow 0$. 

\begin{equation}
\label{eq:NUM_PD_controller}
\begin{array}{ll}
&\mu_l(k+1)= \mu_l(k)  +\beta \big( \mu_l(k) - \mu_l(k-1) \big)\\
&\mu_l(k+1)-\mu_l^\star= \mu_l(k) -\mu_l^\star +\beta \Big( \big( \mu_l(k) -\mu_l^\star \big)- \big( \mu_l(k-1)-\mu_l^\star \big) \Big)\\
&e^\mu_l(k+1)= e^\mu_l(k)  +\beta \big( e^\mu_l(k) - e^\mu_l(k-1) \big) ,
\end{array}
\end{equation}
where $\mu_l^\star$ is the optimal price of link $l$ and $e^\mu_l(k) \triangleq \mu_l(k) -\mu_l^\star $ is the distance between the current price and the optimal price of link $l$. It is easy to note that \eqref{eq:NUM_PD_controller} corresponds to a PD controller for driving the price of link $l$ to its optimal value. Hence, it is obvious that asymptotically~\eqref{eqn:num_dual_multistep} behaves like a PD controller.

To compare the gradient iterations with the multi-step congestion control mechanism, we present representative results from a network with $10$ links and $20$ flows which satisfies Assumption~\ref{ass:num}. The utility functions are on the form $-(M_s-x_s)^2/2$ and $m_s=0$ and $M_s=10^5$ for all sources. As shown in Figure~\ref{fig:num_acceleration}, substantial speedups are obtained.

As a final remark, note that Lemma~\ref{lem:eigenvalue_bouds} underestimates $\lambda_1$ and overestimates $\lambda_n$, so we have no formal guarantee that the multi-step method will always outperform the gradient-based algorithm. However, in our experiments with a large number of randomly generated networks, the disadvantageous situation identified in Section~\ref{sec:robustness} never occurred.

\section{Conclusions}
\label{sec:conclusion}
We have studied accelerated gradient methods for network-constrained optimization problems. In particular, given the bounds of the Hessian of the objective function and the Laplacian of the underlying communication graph, we derived primal and dual multi-step techniques that allow to improve the convergence factors significantly compared to the standard gradient-based techniques. We derived optimal parameters and convergence factors, and characterized the robustness of our tuning rules to errors that occur when critical problem parameters are not known but have to be estimated. Our multi-step techniques were applied to three classes of problems: distributed resource allocation under a network-wide resource constraint, distributed average consensus, and Internet congestion control. We demonstrated, both analytically and in numerical simulations, that the approaches developed in this paper outperform, and often significantly outperforms, alternatives from the literature.
\appendices
\begin{center}
Appendix
\end{center}
\dspace
\subsection{Proof of Theorem \ref{thm:multi_step}}
\label{proof:multi_step}
Let $x^\star$ be the optimizer of~\eqref{eqn:problem_formulation}.
The Taylor series expansion of $\nabla f\big( x(k)\big)$ around $x^{\star}$ yields
\begin{align*}
W \nabla f\big(x(k)\big) &\cong W (\nabla f(x^\star)+\nabla ^2 f(x^{\star})(x(k)-x^{\star}))\\
&= W \nabla ^2 f(x^{\star})(x(k)-x^{\star})
\end{align*}
since $W\nabla f(x^{\star})=0$ by ~\eqref{eqn:fixed_point_feasibility_cond} and \eqref{eqn:weight_constraints}. Introducing
\begin{align*} z(k) \triangleq [  x(k)-x^{\star},\, x(k-1)-x^{\star} ]^\top, \end{align*}
we can thus re-write (\ref{eqn:multistep_scaled_gradient}) as
\begin{equation}
z(k+1) = \underset{\Gamma}{\underbrace{\begin{bmatrix}
B& -\beta I \\ I& 0 \end{bmatrix}
 }}
z(k)+o(z(k)^2),
\label{Amatrix}
\end{equation}
where $B=(1+\beta)I- \alpha W H$ and $ H=\nabla ^2 f(x^{\star})$. Now, for non-zero vectors $v_1$ and $v_2$, consider the eigenvalue equation
\begin{equation}
 \left[\begin{IEEEeqnarraybox*}[][c]{,c/c,}
B & -\beta I \\ I& 0  \end{IEEEeqnarraybox*}\right]
 \left[ \begin{array}{c} v_1 \\ v_2 \end{array}\right]
= \lambda (\Gamma) \left[ \begin{array}{c} v_1 \\ v_2 \end{array}\right]
 \nonumber
 \end{equation}
Since $v_1=\lambda(\Gamma) v_2$, the first row can be re-written as
\begin{equation}
\label{eq:AcharateristicPolynom}
\left( - \lambda^2(\Gamma)I+\lambda(\Gamma) B - \beta I \right)v_2=0.
\end{equation}
Note that (\ref{eq:AcharateristicPolynom}) is a polynomial in $B$ and $B$ is in turn a polynomial in $WH$. Hence, if $\mu$ and $\lambda$ denote the eigenvalues of $B$ and $WH$, respectively, we have
\begin{equation}
\lambda^2(\Gamma) - \left(1+\beta - \alpha \lambda\right) \lambda(\Gamma) + \beta = 0.
\label{eq:charachteristic}
\end{equation}
 The roots of (\ref{eq:charachteristic}) have the form
\begin{equation}
\lambda(\Gamma) =  \frac{1+\beta - \alpha \lambda \pm \sqrt{\Delta}}{2},
\quad
\Delta = \left(1+\beta - \alpha \lambda\right)^2 - 4 \beta.
\label{eq:roots}
\end{equation} If $\Delta \geq 0$, then $|\lambda(\Gamma)|<1$ is equivalent to
\begin{align*}
\nonumber
&\left(1+\beta - \alpha \lambda\right)^2 - 4 \beta \geq 0&\\
&-2< 1+\beta -\alpha \lambda \pm \sqrt{(1+\beta-\alpha \lambda)^2 - 4\beta}<2.&
\end{align*}
which, after simplifications, yield
\[
0<\alpha <2(1+\beta)/\lambda.\]
On the other hand, if $\Delta < 0$, then $|\lambda(\Gamma)|<1$ is equivalent to
\begin{align*}
0 \le \frac{(1+\beta - \alpha \lambda)^2 - \Delta}{4}<1,
\end{align*}
which, after similar simplifications, implies that $0\le \beta <1$.

Note that the upper bound for $\alpha$ gives a necessary condition for $\lambda$. Here we find an upper bound for this eigenvalue. Since $H$ is a positive diagonal matrix, under similarity equivalence we have $WH \sim H^{1/2}WH H^{-1/2}=H^{1/2}WH^{1/2}$.
Without loss of generality assume $x\in \mathbf{R}^{n}$ and $x^\top  x=1$, Then~$x^\top WH x = x^\top  H^{1/2}WH^{1/2}x=y^\top  Wy$,
where~$y=H^{1/2}x$. Clearly, for $y^\top Wy$ it holds that
\begin{align*}
\lambda_1(W)y^\top y\leq y^\top Wy \leq \lambda_{n}(W)y^\top y.
\end{align*}
Now, ~$l\leq y^{\top}y=x^\top Hx\leq u$, implies~$l\lambda_1(W)\leq x^\top  WH x \leq u \lambda_{n}(W)$.
and hence, a sufficient condition on $\alpha$ reads
\begin{align}
0<\alpha< \frac{2(1+\beta)}{u\lambda_{n}(W)}.
\label{eq:alphaBound_sufficient}
\end{align}

Having proven the sufficient conditions for convergence stated in the theorem, we now proceed to estimate the convergence factor. To this end, we need the following lemmas describing the eigenvalue characteristics of $WH$ and $\Gamma$.

\begin{lemma}
\label{lem:multi_step_assist1}
 If $W$ has $m<n$ zero eigenvalues, then $WH$ has exactly $n-m$ nonzero eigenvalues,
 i.e.~$\lambda_1(WH)= \cdots = \lambda_m(WH)  = 0$, $ \lambda_i(WH) \neq 0 \quad  i = m+1, \cdots,n.$
\end{lemma}
\begin{proof}
\label{proof_multi_step_assist1}
From \cite{HoJ:85} we know that if and only if all the principal submatrices of a matrix have nonnegative determinants then that matrix is positive semidefinite. Note that the $i$-th principal submatrix of $WH$, $WH_{i}$, is obtained by multiplication of the corresponding principal submatrix of $W$, $W_{i}$ by the same principal submatrix of $H$, $H_{i}$ from the right, and we have
$\det( WH_{i})=\det(W_{i})\det(H_{i}).$
We know $\det(H_{i})> 0$ and $\det(W_{i})\geq 0$ because $W\geq 0$, thus $\det(WH_{i})\geq 0$ and $WH $ is positive semidefinite. Furthermore
$\mbox{rank}(WH)=\mbox{rank}(W)$.
So $\mbox{rank}(WH)=n-m$ and it means that $WH $ has exactly $m$ zero eigenvalues.
\end{proof}
\begin{lemma}
\label{lem:multi_step_assist2}
 For any $WH$ such that  $\lambda_i(WH) = 0 \mbox{ for }  i=1,\cdots, m$, and $\lambda_i(WH) \neq 0, \mbox{ for }  i = m+1, ...,n.$, the matrix $\Gamma$ has $m$ eigenvalues equal to $1$ and
 the absolute values of the rest of the $2n-m$ eigenvalues are strictly less than $1$.
\end{lemma}
\begin{proof}
For complex $\lambda_i(\Gamma)$ we have $|\lambda_i(\Gamma)|=\beta<1$. For real-valued $\lambda_i(\Gamma)$, on the other hand, the bound on $\alpha$ implies that $\alpha(\lambda (WH))$ is a decreasing function of $\lambda$. In this case, $0<\alpha< \frac{2(1+\beta)}{\overline{\lambda}}$ guarantees that $0 < \alpha < \frac{2(1+\beta)}{\lambda_i(WH)}$
for any $0< \lambda_i(WH) \le \overline{\lambda} $. Note that if we set a tighter bound on $\alpha$, then it does not change satisfactory
condition for having $\left| \lambda(\Gamma) \right|<1$. Only when $ \lambda_i(WH) = 0$,
we have $\lim_{x \rightarrow 0} \alpha = \infty$. For this case, if we substitute $\lambda_i (WH) = 0$
in (\ref{eq:charachteristic}) we obtain $\lambda_{2i-1}(\Gamma)= 1$ and $\lambda_{2i}(\Gamma) = \beta<1$.
\end{proof}

We are now ready to prove the remaining parts of Theorem~\ref{thm:multi_step}. By the Lemmas above, $\Gamma$ has $m<n$ eigenvalues equal to $1$, which correspond to the $m$ zero eigenvalues of $W$ implied by the optimality condition~\eqref{eqn:weight_constraints}. Hence, minimizing $m+1$-th largest eigenvalue of (\ref{Amatrix}) leads
to the optimum convergence factor of the multi-step weighted gradient iterations~\eqref{eqn:multistep_scaled_gradient}.
Calculating $\overline{\lambda}_{\Gamma}\triangleq \underset{\alpha , \beta}{\min} \, \underset{1 \leq j \leq 2n-m}{\max} \left|\lambda_{j}(\Gamma)\right|$
yields the optimum $\alpha^{\star}$ and $\beta^{\star}$. Considering that (\ref{eq:roots}) are the eigenvalues of $\Gamma$,
\begin{align*}
\overline{\lambda}_{\Gamma} = \frac{1}{2} \max \left\{ \left\vert 1+\beta - \alpha \lambda_i\right\vert  +\sqrt{(1+\beta - \alpha \lambda_i)^2-4\beta} \right\},
\end{align*} where $\lambda_i\triangleq \lambda_i(WH),\, \forall i=m+1,..,n$. There are two cases:

Case 1: $\left(1+\beta - \alpha \lambda_i\right)^2- 4 \beta \geq 0$. Then, $a$ and $b$ are non-negative and real with $a \geq b$. Hence, $a^2-b^2 \geq (a-b)^2$ and consequently $a+\sqrt{a^2-b^2} \geq 2a -b \geq b$.

Case 2: $(1+\beta - \alpha \lambda_i)^2- 4 \beta < 0$. In this case, $\lambda_i(\Gamma)$ is complex-valud. Consider $c, d \in \mathbf{R}^+$ with $c< d$. Then,  $\vert c+\sqrt{c^2-d}\vert = \sqrt{c^2-c^2+d} = \sqrt{d} \geq 2c-\sqrt{d}$.

If we substitute these results into $\overline{\lambda}_{\Gamma} $ with $a=1+\beta -\alpha \lambda _i$, $b=2\sqrt{\beta}$ , $c=\vert 1+\beta-\alpha \lambda_i\vert$ and $d=4\beta$  we get
\begin{align*}
\overline{\lambda}_{\Gamma} \geq \max\left\{ \sqrt{\beta}, \max\left\{\vert 1+\beta -\alpha \lambda_i\vert -\sqrt{\beta}\right\}\right\},
\end{align*}
which can be expressed in terms of $\underline{\lambda}$ and $\overline{\lambda}$:
\begin{equation}
\overline{\lambda}_{\Gamma} \geq \max\left\{\sqrt{\beta}, \vert 1+\beta -\alpha \underline{\lambda}\vert -\sqrt{\beta}, \vert 1+\beta -\alpha \overline{\lambda}\vert -\sqrt{\beta} \right\}.
\label{extreme}
\end{equation}
It can be verified that
\begin{align}
\begin{array}{ll}
\max  &\left\{\vert 1+\beta -\alpha \underline{\lambda}\vert-\sqrt{\beta}, \vert 1+\beta -\alpha \overline{\lambda}\vert -\sqrt{\beta} \right\} \\
&\geq \vert 1+\beta-\alpha^{\prime}\underline{\lambda}\vert -\sqrt{\beta},
\end{array}
\label{extreme2}
\end{align}
where $\alpha^{\prime}$
is such that~$\left|1+\beta -\alpha^{\prime} \underline{\lambda}\right|= \left|1+\beta -\alpha^{\prime} \overline{\lambda}\right|$, \emph{i.e.}
\begin{equation}
\alpha^{\prime} = \frac{2(1+\beta)}{\underline{\lambda}+\overline{\lambda}}.
\label{alpha'}
\end{equation}
From (\ref{extreme}), (\ref{extreme2}) and (\ref{alpha'}), we thus obtain
\begin{equation}
\overline{\lambda}_{\Gamma} \geq \max \left\{ \sqrt{\beta},(1+\beta)\frac{\overline{\lambda}-\underline{\lambda}}{\overline{\lambda}+\underline{\lambda}} -\sqrt{\beta}\right\}.
\label{maxBeta}
\end{equation}
Again, the max-operator can be bounded from below by its value at the point where the arguments are equal. To this end, consider $\beta^{\prime}$ whcih satisfies
\begin{equation}
 \sqrt{\beta^{\prime}} = (1+\beta^{\prime})\frac{\overline{\lambda}-\underline{\lambda}}{\overline{\lambda}+\underline{\lambda}} -\sqrt{\beta^{\prime}},
 \nonumber
\end{equation}
that is,
\begin{equation}
 \beta^{\prime} = \left(\frac{\sqrt{\overline{\lambda}}-\sqrt{\underline{\lambda}}}{\sqrt{\overline{\lambda}}+\sqrt{\underline{\lambda}}}\right)^2.
 \label{beta'}
\end{equation}
Since
\begin{equation}
 \max \left\{ \sqrt{\beta}, (1+\beta)\frac{\overline{\lambda}-\underline{\lambda}}{\overline{\lambda}+\overline{\lambda}}-\sqrt{\beta} \right\} \geq \sqrt{\beta^{\prime}}, \label{braceBound}
\end{equation}
we can combine (\ref{braceBound}) and (\ref{maxBeta}) to conclude that
\begin{equation}
\overline{\lambda}_{\Gamma} \geq \sqrt{\beta^{'}} = \frac{\sqrt{\overline{\lambda}}-\sqrt{\underline{\lambda}}}{\sqrt{\overline{\lambda}}+\sqrt{\underline{\lambda}}}
\label{secondlargestbound}
\end{equation}
Our proof is concluded by noting that equality in (\ref{secondlargestbound}) is attained for the smallest non-zero eigenvalue of $\Gamma$ and the optimal step-sizes $\beta^\star$ and $\alpha^{\star}$ stated in the body of the theorem.
\dspace
\subsection{Proof of Proposition~\ref{prop:multistep_restricted}}
\label{proof:prop:multistep_restricted}
As shown in the proof of Theorem 1, the eigenvalues of $WH$ are equal to those of $H^{1/2}WH^{1/2}$. According to \cite[p.225]{HoJ:85} for matrices $W$ and $H^{1/2}WH^{1/2}$, there exists a nonnegative real number $\theta_k$ such that $\lambda_1(H)\leq \theta_k \leq \lambda_n(H)$ and $\lambda_k(H^{1/2}WH^{1/2})=\theta_k \lambda_k(W)$. Letting $k=m+1$ and $k=n$, yields
  $\underline{\lambda} \geq l \underline{\lambda}_W$ and $\overline{\lambda} \leq u \overline{\lambda}_W$. The rest of the proof is similar to that of Theorem \ref{thm:multi_step} and is omitted for brevity.
\dspace
\subsection{Proof of Proposition~\ref{prop:multistep_condition_number_minimization}}
\label{proof:prop:multistep_condition_number_minimization}
Direct calculations yield~$ q^{\star}=(\sqrt{\overline{\lambda}}-\sqrt{\underline{\lambda}})/(\sqrt{\overline{\lambda}}+\sqrt{\underline{\lambda}})  =  1-2/((\overline{\lambda}/\underline{\lambda})^{1/2}+1)$.
Similarly,~$ \widetilde{q}=1-2/((\overline{\lambda}_W/\underline{\lambda}_W)^{1/2}+1)$.
Hence, minimizing $q^\star$ and $\widetilde{q}$ are equivalent to minimizing the condition number of $WH$ and $W$, respectively. 
\dspace
\subsection{Proof of Proposition~\ref{prop:weight_optimization}}
\label{proof:prop:weight_optimization}
Similar to the proof of Theroem~\ref{thm:multi_step} it can be seen that the eigenvalues of $\omega H$ are equal to the ones of $\Omega\triangleq H^{1/2}\omega H^{1/2}$. To have the $m$ zero eigenvalues of $\Omega$ corresponding to the condition $WA^\top = \textbf{0}$ in~\eqref{eqn:weight_constraints}, one needs to condition $V$ in~\eqref{eqn:weight_optimization} to belong to the kernel of $W H^{1/2}$.
Moreover, to restrict the search of $\omega$ to the nonzero eigenspace of $W$, we should have~$x^\top \Omega x > 0$ for all nonzero $x\in V^\bot$. This condition is equivalent to having $y^\top P^\top \Omega P y >0$ for all nonzero $y\in \mathbf{R}^{n}$ and $P$ being the matrix of vectors spanning $V^\bot$.
\dspace
\subsection{Proof of Lemma \ref{lem:primal_dual}}
\label{proof:lem:primal_dual}
To prove (a) we exploit the equivalence of $l$-strong convexity of $f(\cdot)$ and $1/l$-Lipschitz continuity of $\nabla f_{\star}$. Specially according to ~\cite[Theorem 4.2.1]{UrL:96}, for nonzero $z_1,z_2 \in \mathbf{R}^n $,  Lipschitz continuity of $\nabla f_{\star}$ implies that
 \begin{align}
 \nonumber
 \langle \nabla f_{\star}(z_1)-\nabla f_{\star}(z_2), z_1-z_2 \rangle \leq \frac{1}{l}\|z_1-z_2\|^2
 \end{align}
Now, for $-\nabla d(z)= -A\nabla f_{\star}(-A^\top z)+b$, change the right hand side of above inequality to have
\begin{align}
\nonumber
&\langle -\nabla d(z_1)+\nabla d(z_2), z_1-z_2 \rangle\\
\nonumber
&=\langle \nabla f_{\star}(-A^\top z_1)-\nabla f_{\star}(-A^\top z_2), -A^\top (z_1-z_2)\rangle.
\end{align}
In light of $1/l$-Lipschitzness of $\nabla f^{\star}$ we get
\begin{align}
\nonumber
&\langle \nabla f_{\star}(-A^\top z_1)-\nabla f_{\star}(-A^\top z_2), -A^\top (z_1-z_2)\rangle  \\
\nonumber
&\leq \frac{1}{l} \|-A^\top  (z_1-z_2)\|^2  \leq \frac{\lambda_{n}(AA^\top )}{l}  \|z_1-z_2\|^2.
\end{align}
(b) According to ~\cite[Theorem 4.2.2]{UrL:96}, If  $\nabla f(\cdot)$ is $u$-Lipschitz continuous then $f_\star$ is $1/u$-strongly convex, \ie, for non-identical $z_1,z_2 \in \mathbf{R}^n$
\begin{align}
\nonumber
\langle \nabla f_{\star}(z_1)-\nabla f_{\star}(z_2), z_1-z_2 \rangle \geq \frac{1}{u} \|z_1-z_2\|^2
\end{align}
One can manipulate above inequality as
\begin{align}
\nonumber
&\langle -\nabla d(z_1)+\nabla d(z_2), z_1-z_2 \rangle\\
\nonumber
&= \langle \nabla f_{\star}(-A^\top z_1)-\nabla f_{\star}(-A^\top z_2), -A^\top (z_1-z_2) \rangle \\
 \nonumber
 &\geq \frac{1}{u} \|-A^\top (z_1-z_2)\|^2 \geq \frac{\lambda_1(AA^\top)}{u} \|z_1-z_2\|^2.
\end{align}
It is worth noting that here we assume that $A$ is row full rank.
\dspace
\subsection{Proof of Theorem~\ref{thm:dual_convergence}}
\label{proof:thm_dual_convergence}
The result follows from Lemma~\ref{lem:primal_dual} and Theorem~\ref{thm:multi_step} with  $W=I$  and noting that $(\lambda_1(AA^\top)/u) I\leq H \leq (\lambda_n(AA^\top)/l) I$.
\dspace
\subsection{Proof of Lemma~\ref{lem:weighted_gradient_convergence}}
\label{proof:lem:Weighted_gradient_convergence}
Since $f$ is twice differentiable on $[x^{\star}, x]$, we have
\begin{align*}
\nabla f\big(x\big) &= \nabla f(x^\star) + \int_0^1 {\nabla^2 f(x^\star+ \tau(x-x^\star))(x-x^\star)} d\tau \\
& =  A^\top \mu^\star + H(x) (x-x^\star),
\end{align*}
where we have used the fact that $\nabla f(x^{\star})=A^{\top}\mu^{\star}$ and introduced $H(x)=\int_0^1{\nabla^2 f(x^\star+ \tau(x-x^\star))}d\tau$. By virtue of Assumption~\ref{ass:objective_bounds}, $H(x)$ is symmetric and nonnegative definiteand satisfies $l I \leq H(x) \leq u I$~\cite{polyak} . Hence from~\eqref{eqn:scaled_gradient_iteration} and~\eqref{eqn:weight_constraints}
\begin{align*}
&\Vert x(k+1) - x^\star \Vert = \Vert x(k)-x^\star - \alpha W \nabla f\big(x(k)\big) \Vert  \\
&=\Vert x(k) - x^\star - \alpha W (A^\top \mu^\star + H(x(k))(x(k)-x^\star)) \Vert \\
&= \Vert (I - \alpha W H(x(k))) (x(k)-x^\star) \Vert\\
& \leq \Vert I - \alpha W H(x(k)) \Vert \Vert x(k)-x^\star \Vert.
\end{align*}
The rest of the proof follows the same steps as~\cite[Theorem~3]{polyak}. Essentially for fixed step-size $0<\alpha < 2/\overline{\lambda}$, the iterations in ~\eqref{eqn:scaled_gradient_iteration} converge linearly with factor $q_2=\mbox{max} \{ \vert 1-\alpha \underline{\lambda} \vert, \vert 1-\alpha \overline{\lambda}\vert\}$. The minimum convergence factor $q_{G}^\star = \frac{\overline{\lambda}-\underline{\lambda}}{\underline{\lambda}+\overline{\lambda}}$ is obtained by minimizing $q_{G}$ over $\alpha$, which yields the optimal step-size $\alpha^\star = \frac{2}{\underline{\lambda}+\overline{\lambda}}$.
\dspace
\subsection{Proof of Proposition~\ref{prop:robustness_stability_conditions}}
\label{proof:prop:robustness_stability_conditions}
According to Lemma~\ref{lem:weighted_gradient_convergence}, the weighted gradient iterations~\eqref{eqn:scaled_gradient_iteration} with estimated step-size $\oest{\alpha} = 2/(\uest{\lambda}+\oest{\lambda})$ will converge provided that~$0<\oest{\alpha}<2/\overline{\lambda}$, \emph{i.e.} when $\overline{\lambda}< \uest{\lambda}+\oest{\lambda}$.

For the multi-step algorithm~\eqref{eqn:multistep_scaled_gradient}, Theorem~\ref{thm:multi_step} guarantees convergence if $0\leq \oest{\beta}<1, \, 0<\oest{\alpha}<2(1+\oest{\beta})/\overline{\lambda}.$ The assumption $0<\uest{\lambda}\leq \oest{\lambda}$ implies that the condition on $\oest{\beta}$ is always satisfied. Regarding $\oest{\alpha}$, inserting the expression for $\oest{\beta}$ in the upper bound for $\oest{\alpha}$ and simplifying yields
\[
\frac{4}{
\left(\sqrt{\uest{\lambda}}+\sqrt{\oest{\lambda}}\right)^2} <
2\frac{2(\oest{\lambda}+\uest{\lambda)}}{\left(\sqrt{\oest{\lambda}}+\sqrt{\uest{\lambda}}\right)^2}
\frac{1}{\overline{\lambda}}
%
\]
which is satisfied if  $0<\overline{\lambda}<\oest{\lambda}+\uest{\lambda}$. The statement is proven.
\dspace
\subsection{Proof of Lemma~\ref{lem:conv_factors_inaccurate}}
\label{proof:lem:conv_factors_inaccurate}
We consider two cases. First, when $\uest{\lambda}+\oest{\lambda}<\underline{\lambda}+\overline{\lambda}$ combined with the assumption that
$0<\overline{\lambda}<\uest{\lambda}+\oest{\lambda}$ yields $\oest{\alpha}\overline{\lambda}>1$, which means that $\vert 1-\oest{\alpha}\overline{\lambda}\vert = \oest{\alpha}\overline{\lambda}-1$.  Moreover, $\oest{\alpha}\overline{\lambda}-1\geq 1 - \oest{\alpha}\underline{\lambda}$, so by Lemma~\ref{lem:weighted_gradient_convergence}
\begin{align*}
\oest{q}_G&=
\mbox{max}\{\oest{\alpha}\overline{\lambda}-1, \max\{1-\oest{\alpha}\underline{\lambda}, \oest{\alpha}\underline{\lambda}-1\}\}=
\oest{\alpha}\overline{\lambda}-1\\
&=2\overline{\lambda}/(\uest{\lambda}+\oest{\lambda})-1.
\end{align*}

The second case is when $\uest{\lambda}+\oest{\lambda}>\underline{\lambda}+\overline{\lambda}$. Then,  $\oest{\alpha}\underline{\lambda}<1$ and hence $|1-\oest{\alpha}\underline{\lambda}|=1-\oest{\alpha}\underline{\lambda}$. Moreover, $1-\oest{\alpha}\underline{\lambda}\geq \oest{\alpha}\overline{\lambda}-1$, so
\begin{align*}
\oest{q}_G&=
\mbox{max}\{1-\oest{\alpha}\underline{\lambda},\max\{ \oest{\alpha}\overline{\lambda}-1, 1-\oest{\alpha}\overline{\lambda}\}\}
=1-\oest{\alpha}\underline{\lambda}\nonumber
\end{align*}

The convergence factor of the multi-step iterations with inaccurate step-sizes~\eqref{eqn:multistep_step_inaccurate} follows directly from Theorem~\ref{thm:multi_step}.
\dspace
\subsection{Proof of Proposition~\ref{prop:conv_factors_inaccurate}}
\label{proof:prop:conv_factors_inaccurate}
We analyze the four quadrants ${\mathcal Q}_1$ through ${\mathcal Q}_4$ in order.
\begin{enumerate}
\item[${\mathcal Q}_1:$] when $(\uest{\varepsilon},\oest{\varepsilon}) \in {\mathcal Q}_1$ we have
    $\uest{\lambda}>\underline{\lambda}$ and $\oest{\lambda}>\overline{\lambda}>\overline{\lambda}$. From convergence factor of multi-step weighted gradient method given in~\eqref{eqn:multistep_inaccurate_factor} it then follows that
    \begin{align*}
    \oest{q}= 1+\oest{\beta}-\oest{\alpha}\underline{\lambda}-\oest{\beta}^{1/2}.
    \end{align*}
Moreover, since in this quadrant $\oest{\lambda}+\uest{\lambda}\geq \overline{\lambda}+\underline{\lambda}$, from~\eqref{eqn:gradient_inaccurate_factor} we have $\oest{q}_G=1-2\underline{\lambda}/(\uest{\lambda}+\oest{\lambda})$. A direct comparison between the two expressions yields that $\oest{q}\leq \oest{q}_G$.
\item[${\mathcal Q}_2:$] when $(\uest{\varepsilon},\oest{\varepsilon}) \in {\mathcal Q}_2$ we have $\underline{\lambda}<\uest{\lambda}$ and $\oest{\lambda} < \overline{\lambda}$. Combined with the stability assumption $\uest{\lambda}+\oest{\lambda} > \overline{\lambda}$, straightforward calculations show that the convergence factor of the multi-step iterations with inaccurate step-sizes~\eqref{eqn:multistep_step_inaccurate} is
 \begin{align*}
 \oest{q}=
\left\{
\begin{array}{ll}
\oest{\alpha}\overline{\lambda} - \oest{\beta} - 1-\sqrt{\oest{\beta}} & \oest{\lambda}+\oest{\lambda} \leq \underline{\lambda}+\overline{\lambda},\\
1+\oest{\beta}-\oest{\alpha}\underline{\lambda}-\sqrt{\oest{\beta}} & \mbox{otherwise},
\end{array}\right.
 \end{align*}
 Moreover, for this quadrant the convergence factor of weighted gradient method is given by~\eqref{eqn:gradient_inaccurate_factor}.
To verify that $\oest{q}<\oest{q}_G$ we perform the following comparisons:

(a) \label{proof:hb_better_left_corner_item2} If $\uest{\lambda}+ \oest{\lambda} < \underline{\lambda} + \overline{\lambda}$ then we have $\oest{q}=\oest{\alpha}\overline{\lambda}-\oest{\beta}-1-\oest{\beta}^{1/2}$ and $\oest{q}_G = (2\overline{\lambda})/(\uest{\lambda}+ \oest{\lambda})-1$. To show that $\oest{q}<\oest{q}_G$ we rearrange it to obtain the following inequality
\begin{align*}
\Delta \triangleq (\overline{\lambda}-\oest{\lambda}+\oest{\lambda}^{1/2}\uest{\lambda}^{1/2})(\oest{\lambda}+\uest{\lambda})-2\overline{\lambda}\oest{\lambda}^{1/2}\uest{\lambda}^{1/2}<0.
\end{align*}
Further simplifications yield
\begin{align*}
 \Delta &=(\oest{\lambda}+\uest{\lambda} - 2 (\oest{\lambda}\uest{\lambda})^{1/2})\overline{\lambda} -(\oest{\lambda}-(\oest{\lambda}\uest{\lambda})^{1/2})(\oest{\lambda}+\uest{\lambda})  \\
    &=(\oest{\lambda}^{1/2}-\uest{\lambda}^{1/2})^2 \overline{\lambda} - \oest{\lambda}^{1/2}(\oest{\lambda}^{1/2}-\uest{\lambda}^{1/2}) (\oest{\lambda}+\uest{\lambda})\\
    &=(\oest{\lambda}^{1/2}-\uest{\lambda}^{1/2})\left((\oest{\lambda}^{1/2}-\uest{\lambda}^{1/2})\overline{\lambda} - \oest{\lambda}^{1/2}(\oest{\lambda}+\uest{\lambda})\right)\\
    &=(\oest{\lambda}^{1/2}-\uest{\lambda}^{1/2})\left(-\oest{\lambda}^{1/2}(\oest{\lambda}+\uest{\lambda}-\overline{\lambda}) - \uest{\lambda}^{1/2}\overline{\lambda}\right) < 0
   \end{align*}

    Note that the negativity of above quantity comes from the stability condition, $\oest{\lambda}+\uest{\lambda}>\overline{\lambda}$.

(b) If $\uest{\lambda}+ \oest{\lambda} > \underline{\lambda} + \overline{\lambda}$ then we have $\oest{q} = 1+\oest{\beta}-\oest{\alpha}\underline{\lambda} -(\oest{\beta})^{1/2}$ and $\oest{q}_G =1-(2\underline{\lambda})/(\uest{\lambda}+\oest{\lambda})$. After some simplifications, we see that $\oest{q}< \oest{q}_G$ boils down to the inequality
    $-(\uest{\lambda}+\oest{\lambda})\uest{\lambda}^{1/2}\oest{\lambda}^{1/2}+2\underline{\lambda}\uest{\lambda}^{1/2}\oest{\lambda}^{1/2}-\underline{\lambda}(\uest{\lambda}+\oest{\lambda})< 0 $ or equivalently $-(\uest{\lambda}+\oest{\lambda}-2\underline{\lambda})\uest{\lambda}^{1/2}\oest{\lambda}^{1/2} -\underline{\lambda}(\uest{\lambda}+\oest{\lambda})<0$ which holds by noting that $\uest{\lambda}+\oest{\lambda}>\underline{\lambda}+\overline{\lambda}>2\underline{\lambda}$.

(c) for the case $\uest{\lambda}+ \oest{\lambda} = \underline{\lambda} + \overline{\lambda}$, we have~$\oest{q}=1+\oest{\beta}-\oest{\alpha}\underline{\lambda} -(\oest{\beta})^{1/2}$ and $\oest{q}_G=(\overline{\lambda}-\underline{\lambda})/(\underline{\lambda}+\overline{\lambda})$ which coincides with the optimal convergence factor of unperturbed gradient method.
 After some rearrangements we notice that $\oest{q}<\oest{q}_G$ reduces to checking that
\begin{align*}
    (\oest{\lambda}^{1/2}-\uest{\lambda}^{1/2})(\overline{\lambda}-\underline{\lambda}) < (\uest{\lambda}^{1/2}+\oest{\lambda}^{1/2})(\uest{\lambda}+\oest{\lambda})
\end{align*}
that holds since  $\oest{\lambda}^{1/2}-\uest{\lambda}^{1/2}<\uest{\lambda}^{1/2}+\oest{\lambda}^{1/2}$ and $\overline{\lambda}-\underline{\lambda}<\overline{\lambda}+\underline{\lambda} = \uest{\lambda}+\oest{\lambda}$.

\item [${\mathcal Q}_3:$] if $(\uest{\varepsilon},\oest{\varepsilon}) \in {\mathcal Q}_3$ we have  $0<\uest{\lambda}<\underline{\lambda}$ and $\oest{\lambda}< \overline{\lambda}$. Combined with the stability assumption $\oest{\lambda}+ \uest{\lambda}> \overline{\lambda}$, one can verify that the convergence factors of the two perturbed iterations are
$\oest{q}_G = (2\overline{\lambda})/(\uest{\lambda}+ \oest{\lambda})-1$ and $\oest{q}= \oest{\alpha}\overline{\lambda} - \oest{\beta} - 1-(\oest{\beta})^{1/2}$, respectively. The fact that $\oest{q}< \oest{q}_G$ was proven in step (a) of the analysis of ${\mathcal Q}_2$.
\item [${\mathcal Q}_4:$] if $(\uest{\varepsilon},\oest{\varepsilon}) \in {\mathcal Q}_4$ then,~\eqref{eqn:multistep_inaccurate_factor} implies that~$\oest{q} = \oest{\beta}^{1/2}.$
On the other hand, for this region~\eqref{eqn:gradient_inaccurate_factor} yields $\oest{q}_G=(\overline{\lambda}-\underline{\lambda})/(\underline{\lambda}+\overline{\lambda}).$ To conclude, we need to verify that there exists $\oest{\lambda}$ and $\uest{\lambda}$ such that $\oest{q}>\oest{q}_G$, \emph{i.e.} such that
 $(\oest{\lambda}^{1/2}-\uest{\lambda}^{1/2})/(\oest{\lambda}^{1/2}+\uest{\lambda}^{1/2})>(\overline{\lambda}-\underline{\lambda})/(\underline{\lambda}+\overline{\lambda})$.
We do so by multiplying both sides with $(\underline{\lambda}+\overline{\lambda})(\oest{\lambda}^{1/2}+\uest{\lambda}^{1/2})$ and simplifying to find that the inequality holds if
$\underline{\lambda} \oest{\lambda}^{1/2}>\overline{\lambda} \uest{\lambda}^{1/2}$, or equivalently $\oest{\lambda}/\uest{\lambda}>\overline{\lambda}^2/\underline{\lambda}^2$. The statement is proven.
\end{enumerate}
\dspace
\subsection{Proof of Proposition~\ref{prop:shift_register_optimal_parameters}}
\label{proof:prop:shift_register_optimal_parameters}
The iterations~\eqref{eqn:optimization_based_consensus_multistep} and~\eqref{eqn:shift_register_consensus} are equivalent when
\begin{align*}
	(1-\zeta) &= -\beta\\
	(1+\beta)I-\alpha W &= \zeta Q
\end{align*}
The first condition implies that $\zeta^{\star}=(1+\beta^{\star})$. Combining this expression with the second condition, we find
\begin{align*}
	Q^{\star} &= I-\frac{\alpha^{\star}}{1+\beta^{\star}}W^{\star} = I-\frac{2}{\underline{\lambda}+\overline{\lambda}}W^{\star}
\end{align*}
Noting that for the consensus case, $\underline{\lambda}=\lambda_2(W^{\star})$ and $\overline{\lambda}=\lambda_n(W^{\star})$ concludes the proof.
\dspace
\subsection{Proof of Lemma~\ref{lem:eigenvalue_bouds}}
\label{proof:lem:eigenvalue_bouds}
For the upper bound on $\lambda_n(RR^\top)$, we use a similar approach as \cite[Lemma~3]{low99}. Specially, from~\cite[p.313]{HoJ:85},
\begin{align*}
\lambda^2_n(RR^\top) &= \Vert R R^\top \Vert_2^2\leq \Vert R R^\top \Vert_\infty \Vert R R^\top \Vert_1 = \Vert R R^\top \Vert_\infty^2.
\end{align*}
Hence,
\begin{align*}
& \lambda_n(RR^\top) = \max\limits_l \sum_{l^\prime}[R R^\top ]_{ll^\prime} =\max\limits_l \sum_{l^\prime}\sum_{s}R_{ls}R_{l^\prime s} \\
&\leq \max\limits_l \sum_{s}R_{ls}l_{\max} \leq s_{\max}l_{\max}.
\end{align*}

To find a lower bound on $\lambda_1(RR^\top)$ we consider the definition $\lambda_1(RR^\top) = \underset {\Vert x\Vert_2 =1}{\min} \Vert R^\top x \Vert_2^2$. We have
\begin{align*}
[R^\top x]_{s} = \sum_{l=1}^{L} [R^\top ]_{sl} x_l = \sum_{l=1}^{L} R_{ls} x_l.
\end{align*}
According to Assumption~\ref{ass:num}, $R^\top$ has $L$ independent rows that have only one non-zero (equal to $1$) component. Hence,
\begin{align*}
\Vert R^\top x\Vert_2^2 &= \sum_{s=1}^L x_s^2 + \sum_{s=S-L+1}^{S} {\left(\sum_{l=1}^{L}R_{ls}x_l\right)^2} \\
&=1 + \sum_{s=S-L+1}^{n} {\left(\sum_{l=1}^{L}R_{ls}x_l\right)^2} \geq 1 ,
\end{align*}
where the last equality is due to $\Vert x \Vert_2 = 1$. 





\ifCLASSOPTIONcaptionsoff
  \newpage
\fi



\bibliographystyle{IEEEtran}
%
\bibliography{referencebib}

\begin{thebibliography}{10}
\providecommand{\url}[1]{#1}
\csname url@samestyle\endcsname
\providecommand{\newblock}{\relax}
\providecommand{\bibinfo}[2]{#2}
\providecommand{\BIBentrySTDinterwordspacing}{\spaceskip=0pt\relax}
\providecommand{\BIBentryALTinterwordstretchfactor}{4}
\providecommand{\BIBentryALTinterwordspacing}{\spaceskip=\fontdimen2\font plus
\BIBentryALTinterwordstretchfactor\fontdimen3\font minus
  \fontdimen4\font\relax}
\providecommand{\BIBforeignlanguage}[2]{{%
\expandafter\ifx\csname l@#1\endcsname\relax
\typeout{** WARNING: IEEEtran.bst: No hyphenation pattern has been}%
\typeout{** loaded for the language `#1'. Using the pattern for}%
\typeout{** the default language instead.}%
\else
\language=\csname l@#1\endcsname
\fi
#2}}
\providecommand{\BIBdecl}{\relax}
\BIBdecl

\bibitem{kelly98}
F.~Kelly, A.~Maulloo, and D.~Tan, ``Rate control in communication networks:
  shadow prices, proportional fairness and stability,'' \emph{Journal of the
  Operational Research Society}, vol.~49, pp. 237--252, 1998.

\bibitem{olfati2007}
R.~Olfati-Saber, J.~A. Fax, and R.~M. Murray, ``Consensus and cooperation in
  networked multi-agent systems,'' \emph{Proceedings of the IEEE}, vol. 95
  Issue: 1, pp. 215--233, 2007.

\bibitem{Kar12}
S.~{Kar}, J.~M.~F. {Moura}, and K.~{Ramanan}, ``{Distributed Parameter
  Estimation in Sensor Networks: Nonlinear Observation Models and Imperfect
  Communication},'' \emph{IEEE Trans. on Information Theory}, 2012.

\bibitem{Boyd11}
S.~Boyd, N.~Parikh, E.~Chu, B.~Peleato, and J.~Eckstein, ``Distributed
  optimization and statistical learning via the alternating direction method of
  multipliers,'' \emph{Foundations and Trends in Machine Learning}, vol. 3
  Issue: 1, pp. 1--122, 2011.

\bibitem{CellularPower00}
D.~Goodman and N.~Mandayam, ``Power control for wireless data,'' \emph{Personal
  Communications, IEEE}, vol. 7 Issue:2, pp. 48--54, 2000.

\bibitem{Tsitsiklis86}
J.~Tsitsiklis, D.~Bertsekas, and M.~Athans, ``Distributed asynchronous
  deterministic and stochastic gradient optimization algorithms,''
  \emph{Automatic Control, IEEE Transactions on}, vol. 31 Issue: 9, pp.
  803--812, 1986.

\bibitem{Cao06}
M.~Cao, D.~A. Spielman, and E.~M. Yeh, ``Accelerated gossip algorithms for
  distributed computation,'' in \emph{44th Annual Allerton Conference on
  Communication, Control, and Computation}, 2006, pp. 952--959.

\bibitem{bjornThesis}
B.~Johansson, ``On distributed optimization in networked systems,'' Ph.D.
  dissertation, Royal Institute of Technology, 2008.

\bibitem{Bertsekas:nonlinear}
D.~Bertsekas., \emph{Nonlinear Programming}.\hskip 1em plus 0.5em minus
  0.4em\relax Athena Scientific, 1999.

\bibitem{damiano11}
F.~Zanella, D.~Varagnolo, A.~Cenedese, G.~Pillonetto, and L.~Schenato,
  ``Newton-raphson consensus for distributed convex optimization,'' in
  \emph{IEEE Conference on Decision and Control (CDC)}, 2011.

\bibitem{Asu011}
E.~Wei, A.~Ozdaglar, and A.~Jadbabaie, ``A distributed newton method for
  network utility maximization, i: Algorithm,'' \emph{LIDS report 2832}, 2011.

\bibitem{Bertsekas1989}
D.~Bertsekas and J.~Tsitsiklis, \emph{Parallel and distributed
  computation:Numerical methods}.\hskip 1em plus 0.5em minus 0.4em\relax New
  York: Athena Scientific, 1997.

\bibitem{ADMMConsensus11}
T.~Erseghe, D.~Zennaro, E.~Dall'Anese, and L.~Vangelista, ``Fast consensus by
  the alternating direction multipliers method,'' \emph{Signal Processing, IEEE
  Transactions on}, vol.~59, no.~11, pp. 5523 --5537, nov. 2011.

\bibitem{polyak}
B.~Polyak, \emph{Introduction to Optimization}.\hskip 1em plus 0.5em minus
  0.4em\relax ISBN 0-911575-14-6, 1987.

\bibitem{Nesterov04}
Y.~Nesterov, \emph{Introductory Lectures on Convex Optimization: A Basic
  Course}.\hskip 1em plus 0.5em minus 0.4em\relax Springer, 2004.

\bibitem{Nesterov05}
------, ``Smooth minimization of non-smooth functions,'' \emph{Mathematical
  Programming}, 2005.

\bibitem{Devolder11_Double}
O.~Devolder, F.~Glineur, and Y.~Nesterov, ``A double smoothing technique for
  constrained convex optimization problems and applications to optimal
  control,'' \emph{submitted to SIAM Journal on Optimization}, 2011.

\bibitem{XiB:06}
L.~Xiao and S.~Boyd, ``Optimal scaling of a gradient method for distributed
  resource allocation,'' \emph{J. Opt. Theory and Applications}, vol. 129
  Issue:3, pp. 469--488, 2006.

\bibitem{HSS:80}
Y.~C. Ho, L.~Servi, and R.~Suri, ``A class of center-free resource allocation
  algorithms,'' \emph{Large Scale Systems}, vol.~1, pp. 51--62, 1980.

\bibitem{Van:12}
L.~Vandenberghe, ``Course notes for optimization methods for large-scale
  systems, ee236c, dual decomposition chapter,'' 2012.

\bibitem{HoJ:85}
R.~A. Horn and C.~R. Johnson, \emph{Matrix Analysis}.\hskip 1em plus 0.5em
  minus 0.4em\relax Cambridge, 1985.

\bibitem{jadbabaie03}
A.~Jadbabaie, J.~Lin, and A.~Morse, ``Coordination of groups of mobile
  autonomous agents using nearest neighbor rules,'' \emph{Automatic Control,
  IEEE Transactions on}, vol.~48, pp. 988 -- 1001, 2003.

\bibitem{rabbat010}
B.~Oreshkin, M.~Coates, and M.~Rabbat, ``Optimization and analysis of
  distributed averaging with short node memory,'' \emph{Signal Processing, IEEE
  Transactions on}, vol. 58 Issue: 5, pp. 2850 --2865, 2010.

\bibitem{XiB:04}
L.~Xiao and S.~Boyd, ``Fast linear iterations for distributed averaging,''
  \emph{Systems and Control Letters}, vol. 53 Issue: 1, pp. 65--78, 2004.

\bibitem{young72}
D.~M. Young, ``Second-degree iterative methods for the solution of large linear
  systems,'' \emph{Journal of Approximation Theory}, 1972.

\bibitem{anderson09}
J.~Liu, B.~D.~O. Anderson, M.~Cao, and A.~S. Morse, ``Analysis of accelerated
  gossip algorithms,'' in \emph{48th IEEE Conference on Decision and Control
  (CDC)}, 2009.

\bibitem{GoV:61}
G.~H. Golub and R.~S. Varga, ``Chebyshev semi-iterative methods, successive
  overrelaxation iterative methods, and second order richardson iterative
  methods,'' \emph{Numerische Matematik}, vol.~3, pp. 147--156, 1961.

\bibitem{Chu:97}
F.~R.~K. Chung, \emph{Spectral Graph Theory}.\hskip 1em plus 0.5em minus
  0.4em\relax CBMS Regional Conference Series in Mathematics, No. 92, American
  Mathematical Society, 1997.

\bibitem{low99}
S.~Low and D.~Lapsley, ``Optimization flow control - i: Basic algorithm and
  convergence.'' \emph{IEEE/ACM Transactions on Networking}, vol. 7 Issue: 6,
  pp. 861--874, 1999.

\bibitem{Mikael04}
L.~Xiao, M.~Johansson, and S.~Boyd, ``Simultaneous routing and resource
  allocation via dual decomposition,'' \emph{IEEE Transactions on
  Communications}, vol. 52 Issue: 7, pp. 1136--1144, 2004.

\bibitem{Chaing07}
M.~Chiang, S.~Low, A.~Calderbank, and J.~Doyle, ``Layering as optimization
  decomposition: A mathematical theory of network architectures,''
  \emph{Proceedings of the IEEE}, vol. 95 Issue:1, pp. 255 -- 312, 2007.

\bibitem{UrL:96}
J.~B.~H. Urruty and C.~Lemar\'echal, \emph{Convex Analysis and Minimization
  Algorithms II}.\hskip 1em plus 0.5em minus 0.4em\relax Springer, 1996.

\end{thebibliography}

%
%

%

%
%
%
%



\end{document}